\documentclass[sn-mathphys-num]{sn-jnl}

\usepackage{graphicx}%
\usepackage{multirow}%
\usepackage{amsmath,amssymb,amsfonts}%
\usepackage{amsthm}%
\usepackage{mathrsfs}%
\usepackage[title]{appendix}%
\usepackage{xcolor}%
\usepackage{textcomp}%
\usepackage{manyfoot}%
\usepackage{booktabs}%
\usepackage{listings}%
\usepackage{algorithm,algorithmic}

\theoremstyle{thmstyleone}%
\newtheorem{theorem}{Theorem}
\newtheorem{lemma}{Lemma}
\theoremstyle{thmstyletwo}%
\newtheorem{remark}{Remark}%

\theoremstyle{thmstylethree}%

\begin{document}

\title[On the Convergence Analysis of Yau-Yau Nonlinear Filtering Algorithm]{On the Convergence Analysis of Yau-Yau Nonlinear Filtering Algorithm: from a Probabilistic Perspective}


\author[1]{\fnm{Zeju} \sur{Sun}}\email{szj20@mails.tsinghua.edu.cn}

\author[2]{\fnm{Xiuqiong} \sur{Chen}}\email{cxq0828@ruc.edu.cn}

\author*[1,3]{\fnm{Stephen S.-T.} \sur{Yau}}\email{yau@uic.edu}

\affil[1]{\orgdiv{Department of Mathematical Sciences}, \orgname{Tsinghua
		University}, \orgaddress{\city{Beijing}, \postcode{100084}, \country{China}}}

\affil[2]{\orgdiv{School of Mathematics}, \orgname{Renmin University of China}, \orgaddress{\city{Beijing}, \postcode{100872}, \country{China}}}

\affil[3]{\orgname{Beijing Institute of Mathematical Sciences and Applications}, \orgaddress{\city{Beijing}, \postcode{101408}, \country{China}}}


\abstract{At the beginning of this century, a real time solution of the nonlinear filtering problem without memory was proposed in \cite{yymethod1,yymethod} by the third author and his collaborator, and it is later on referred to as Yau-Yau algorithm. During the last two decades, a great many nonlinear filtering algorithms have been put forward and studied based on this framework. In this paper, we will generalize the results in the original works and conduct a novel convergence analysis of Yau-Yau algorithm from a probabilistic perspective. Instead of considering a particular trajectory, we estimate the expectation of the approximation error, and show that commonly-used statistics of the conditional distribution (such as conditional mean and covariance matrix) can be accurately approximated with arbitrary precision by Yau-Yau algorithm, for general nonlinear filtering systems with very liberal assumptions. This novel probabilistic version of convergence analysis is more compatible with the development of modern stochastic control theory, and will provide a more valuable theoretical guidance for practical implementations of Yau-Yau algorithm.}




\maketitle

\section{Introduction}
Filtering is an important subject in the field of modern control theory, and has wide applications in various scenarios such as signal processing \cite{candy2016bayesian}\cite{roth2017ensemble}, weather forecast \cite{galanis2006applications}\cite{chen2014short}, aerospace industrial \cite{ichard2015random}\cite{sun2019particle} and so on. The core objective of filtering problem is pursuing accurate estimation or prediction to the state of a given stochastic dynamical system based on a series of noisy observations \cite{jazwinski2007stochastic}\cite{fundamentals}. For practical implementations, it is also necessary that the estimation or prediction to the state can be computed in a recursive and real-time manner.

In the filtering problems we consider in this paper, the evolution of state processes, as well as the noisy observations, is governed by the following system of stochastic differential equations,
\begin{equation}
	\left\{\begin{aligned}
		&dX_{t} = f(X_{t})dt + g(X_{t})dV_{t},\quad X_{0} = \xi,\\
		&dY_{t} = h(X_{t})dt + dW_{t},\quad Y_{0} = 0,
	\end{aligned}\right.\quad t\in [0,T], 
	\label{eq:1}
\end{equation}
in the filtered probability space $(\Omega,\mathcal{F},\{\mathcal{F}_{t}\}_{t=0}^{T},P)$, where $T > 0$ is a fixed termination; $X = \{X_{t}:0\leq t\leq T\}\subset\mathbb{R}^{d}$ is the state process we would like to track; $Y = \{Y_{t}:0\leq t\leq T\}\subset\mathbb{R}^{d}$ is the noisy observation to the state process $X$; $\{V_{t}:0\leq t\leq T\}$ and $\{W_{t}:0\leq t\leq T\}$ are mutually independent, $\mathcal{F}_{t}$-adapted, $d$-dimensional standard Brownian motions; $\xi$ is a $\mathbb{R}^{d}$-valued random variable with probability density function $\sigma_{0}(x)$, which is independent of $V_{t}$ and $W_{t}$; $f:\mathbb{R}^{d}\rightarrow\mathbb{R}^{d}$, $g:\mathbb{R}^{d}\rightarrow\mathbb{R}^{d\times d}$ and $h:\mathbb{R}^{d}\rightarrow\mathbb{R}^{d}$ are sufficiently smooth vector- or matrix- valued functions. 

Mathematically, for a given test function $\varphi:\mathbb{R}^{d}\rightarrow\mathbb{R}$, the \textit{optimal} estimation of $\varphi(X_{t})$ based on the historical observations up to time $t$ is the conditional expectation $E[\varphi(X_{t})|\mathcal{Y}_{t}]$, where $\mathcal{Y}_{t} := \sigma\{Y_{s}:0\leq s\leq t\}$ the $\sigma$-algebra generated by historical observations. Such estimation is `\textit{optimal}' in the sense that,
\begin{equation}
	E[\varphi(X_{t})|\mathcal{Y}_{t}] = \mathop{\text{arg min}}\limits_{U\text{ is }\mathcal{Y}_{t}\text{-measurable}}E[(\varphi(X_{t})-U)^{2}].
\end{equation}
Therefore, the main task of filtering problem can be specified into finding efficient algorithms to numerically compute the conditional expectation $E[\varphi(X_{t})|\mathcal{Y}_{t}]$, or equivalently, the conditional probability distribution $P[X_{t}\in\cdot|\mathcal{Y}_{t}]$ or the conditional probability density function (if exists).

With some regularity assumptions on the coefficients $f,g,h$ in the system \eqref{eq:1}, the conditional probability measure $P[X_{t}\in\cdot|\mathcal{Y}_{t}]$ is absolutely continuous with respect to the Lebesgue measure in $\mathbb{R}^{d}$, and the conditional probability density function $\sigma(t,x)$, (without considering a normalization constant), can be described by the well-known DMZ equation, which is named after three researchers: T. E. Duncan \cite{Duncan1967}, R. E. Mortensen \cite{mortensen1966optimal} and M. Zakai \cite{zakai1969optimal}, who derived the equation independently in the late 1960s.  

The DMZ equation satisfied by the unnormalized conditional probability density function $\sigma(t,x)$ is a second-order stochastic partial differential equation, and does not possess an explicit-form solution in general. In their works \cite{yymethod1} and \cite{yymethod}, the third author and his collaborator propose a two-stage algorithm framework to compute the solution of the DMZ equation numerically in a memoryless and real-time manner. Later on, this algorithm is referred to as Yau-Yau algorithm.

The basic idea of Yau-Yau algorithm is that the heavy computation burden of numerically solving a Kolmogorov-type partial differential equation (PDE) can be done off-line, and in the meanwhile, the on-line procedure only consists of the basic computations such as multiplication by the exponential function of the observations. In the framework of Yau-Yau algorithm, various kinds of methods to solving the Kolmogorov-type PDEs, such as spectral methods \cite{luo2013hermite}\cite{dong2020solving}, proper orthogonal decomposition \cite{wang2019proper}, tensor training \cite{li2022solving}, etc, are proposed and applied in specific examples of nonlinear filtering problems. Numerical results in the previous works mentioned above show that Yau-Yau algorithm can provide accurate and real-time estimations to the state process of very general nonlinear filtering problems in low and medium-high dimensional space.

In the original works \cite{yymethod1} and \cite{yymethod}, the convergence analysis of Yau-Yau algorithm is conducted pathwisely. For \textit{regular} paths of observations with some boundedness conditions, it is proved that the numerical solution provided by Yau-Yau algorithm will converge to the exact solution of DMZ equation both pointwisely and in $L^{2}$-sense, as the size of time-discretization step tends to zero, while in the works on the practical implementations of Yau-Yau algorithm, such as \cite{luo2013hermite} and \cite{dong2020solving}, the convergence analysis mainly focuses on the capability of numerical methods to approximate the solution of Kolmogorov-type PDEs arising in Yau-Yau algorithms. 

In this paper, we will revisit the convergence analysis of Yau-Yau algorithm from a probabilistic perspective. Instead of considering the convergence results pathwisely, we prove that the solution of Yau-Yau algorithm will converge to the exact solution of DMZ equation in expectation, and also, after the normalization procedure, the approximated solution to the filtering problem provided by Yau-Yau algorithm will converge to the conditional expectation $E[\varphi(X_{t})|\mathcal{Y}_{t}]$.

The advantage of this probabilistic perspective is that for a theoretically rigorous convergence analysis, instead of regularity assumptions on observation paths, we only need to make assumptions on the coefficients $f,g,h$ of the filtering system \eqref{eq:1} and the test function $\varphi$, which are verifiable off-line in advance for practitioners. In the meanwhile, as shown in the main results of this paper in Section 3, those assumptions we need are in fact quite general, and it is straightforward to check that the most commonly used test functions $\varphi(x) = x_{i}$ and $\varphi(x) = x_{i}x_{j}$, $x = (x_{1},\cdots,x_{d})^{\top}\in\mathbb{R}^{d}$, corresponding to the conditional mean and conditional covariance matrix, as well as the linear Gaussian systems (with $f(x) = Fx$, $g(x)\equiv \Gamma$, and $h(x) = Hx$, $F,H,\Gamma\in \mathbb{R}^{d\times d}$), satisfy all the assumptions. 

Moreover, to the best of the authors' knowledge, most of the theoretical analysis of PDE-based filtering algorithm mainly deals with convergence results with respect to the DMZ equation. In such probabilistic perspective we consider here in this paper, however, it is natural and convenient to make a step forward and discuss the approximation capability of Yau-Yau algorithm to the \textit{normalized} conditional expectation and conditional probability distribution. In this way, we will provide a thorough convergence analysis of the Yau-Yau algorithm for filtering problems.

The organization of this paper is as follows. Section \ref{sec:2} serves as preliminaries, in which we will summarize some basic concepts of filtering problems and the main procedure of Yau-Yau algorithm. The main theorems in this paper will be stated in Section \ref{sec:3}, together with a sketch of the proofs. In the next four sections, we will provide the detailed proofs of the lemmas and theorems. We first focus on the properties of the exact solution of DMZ equation in Section \ref{sec:4} and Section \ref{sec:5}, and then deal with the approximated solutions given by Yau-Yau algorithm in Section \ref{sec:6} and Section \ref{sec:7}. Finally, Section \ref{sec:8} is a conclusion. 

\section{Preliminaries}
\label{sec:2}
In this section, we would like to briefly summarize the theory of nonlinear filtering, including the change-of-measure approach to deriving the DMZ equation, as well as the main idea and procedures of Yau-Yau algorithm. 

In the change-of-measure approach to deriving the DMZ equation corresponding to the filtering system \eqref{eq:1}, we first introduce a series of reference probability measures $\{\tilde{P}_{t}:0\leq t\leq T\}$, absolutely continuous to the original probability measure $P$ with Radon derivatives given by
\begin{equation}
	Z_{t} \triangleq \frac{d\tilde{P}_{t}}{dP}\biggr|_{\mathcal{F}_{t}} = \exp\biggl(-\int_{0}^{t}h(X_{s})^{\top}dW_{s} - \frac{1}{2}\int_{0}^{t}|h(X_{s})|^{2}ds\biggr),\ t\in[0,T].
	\label{eq:3}
\end{equation}

According to Girsanov's theorem, as long as the process $\{Z_{t}:0\leq t\leq T\}$ defined in \eqref{eq:3} is a martingale, then under the reference probability measure $\tilde{P}_{T}$, the observation process $\{Y_{t}:0\leq t\leq T\}$ is a standard Brownian motion which is independent of the state process $X$. 

We also introduce the process $\{\tilde{Z}_{t}:0\leq t\leq T\}$, $\tilde{Z}_{t} = Z_{t}^{-1}$, to be the inverse of $Z_{t}$, which is also a Radon derivative and can be expressed by the stochastic integral with respect to $Y$ as follows:
\begin{equation}
	\tilde{Z}_{t} = Z_{t}^{-1} = \frac{dP}{d\tilde{P}_{t}}\biggr|_{\mathcal{F}_{t}} = \exp\biggl(\int_{0}^{t}h(X_{s})^{\top}dY_{s} - \frac{1}{2}\int_{0}^{t}|h(X_{s})|^{2}ds\biggr),\ t\in[0,T].
	\label{eq:4b}
\end{equation}
Therefore, for any $\mathcal{F}_{t}$-measurable, integrable random variable $U\in\mathcal{F}_{t}$, its expectation with respect to measure $P$ can be computed by
\begin{equation}
	E[U] = \tilde{E}\left[\tilde{Z}_{t}U\right],
\end{equation}
where $\tilde{E}$ means the expectation is taken under the probability measure $\tilde{P}_{T}$. 

As an extension of Bayesian formula in the context of continuous-time stochastic processes, the following Kallianpur-Striebel formula allows us to express and calculate the solution of filtering problem, $E[\varphi(X_{t})|\mathcal{Y}_{t}]$, by a ratio of conditional expectations under $\tilde{P}_{T}$:
\begin{equation}
	E[\varphi(X_{t})|\mathcal{Y}_{t}] = \frac{\tilde{E}\left[\tilde{Z}_{t}\varphi(X_{t})|\mathcal{Y}_{t}\right]}{\tilde{E}\left[\tilde{Z}_{t}|\mathcal{Y}_{t}\right]},\ t\in[0,T]. 
	\label{eq:6a}
\end{equation}
Since the denominator $\tilde{E}\left[\tilde{Z}_{t}|\mathcal{Y}_{t}\right]$ in \eqref{eq:6a} is independent of the test function $\varphi$, people often refer to the nominator, $\tilde{E}\left[\tilde{Z}_{t}\varphi(X_{t})|\mathcal{Y}_{t}\right]$, as the unnormalized conditional expectation of $\varphi(X_{t})$. The corresponding measure-valued stochastic process $\{\rho_{t}:0\leq t\leq T\}$ defined by
\begin{equation}
	\rho_{t}(A):=\tilde{E}\left[\tilde{Z}_{t}1_{A}|\mathcal{Y}_{t}\right],\ \forall\ A\in \mathcal{F}_{t},\ t\in [0,T],
\end{equation}
is also referred to as unnormalized conditional probability measure, and we also denote the unnormalized conditional expectation by
\begin{equation}
	\rho_{t}(\varphi) := \tilde{E}\left[\tilde{Z}_{t}\varphi(X_{t})|\mathcal{Y}_{t}\right],\ \varphi\text{ is a test function},\ t\in[0,T],
\end{equation} 
With sufficient regularity assumptions on the coefficients $f,g,h$ and test function $\varphi$, the evolution of $\rho_{t}(\varphi)$ is governed by the following well-known DMZ equation:
\begin{equation}
	\rho_{t}(\varphi) = \rho_{0}(\varphi) + \int_{0}^{t}\rho_{s}(\mathcal{L}\varphi)ds + \sum_{j=1}^{d}\int_{0}^{t}\rho_{s}(h_{j}\varphi)dY_{s}^{j},\ t\in [0,T]. 
\end{equation}
where
\begin{equation}
	\mathcal{L} = \frac{1}{2}\sum_{i,j=1}^{d}a^{ij}(x)\frac{\partial^{2}}{\partial x_{i}\partial x_{j}} + \sum_{i=1}^{d}f_{i}(x)\frac{\partial}{\partial x_{i}}
\end{equation}
is a second-order elliptic operator with $a(x) := (a^{ij}(x))_{1\leq i,j\leq d} = g(x)g(x)^{\top}$.

If the stochastic measures $\rho_{t}$, $t\in[0,T]$, are almost surely absolutely continuous to the Lebesgue measure in $\mathbb{R}^{d}$, and the density functions (or the Radon derivatives) $\sigma(t,x)$, as well as the derivatives of $\sigma(t,x)$, is square-integrable, then $\sigma(t,x)$ is the solution to the following equation, (which is also referred to as the DMZ equation):
\begin{equation}
	\left\{\begin{aligned}
		&d\sigma(t,x) = \mathcal{L}^{*}\sigma(t,x)dt + \sum_{j=1}^{d}h_{j}(x)\sigma(t,x)dY_{t}^{j},\ t\in[0,T],\\
		&\sigma(0,x) = \sigma_{0}(x),
	\end{aligned}\right.
	\label{eq:2}
\end{equation}
which is a second-order stochastic partial differential equation with 
\begin{equation}
	\mathcal{L}^{*}(\star) = \frac{1}{2}\sum_{i,j=1}^{d}\frac{\partial^{2}}{\partial x_{i}\partial x_{j}}(a^{ij}\star) - \sum_{i=1}^{d}\frac{\partial}{\partial x_{i}}(f_{i}\star). 
\end{equation}
the adjoint operator of $\mathcal{L}$. 

In this case, the unnormalized conditional expectation $\rho_{t}(\varphi)$ can be expressed as 
\begin{equation}
	\rho_{t}(\varphi) = \tilde{E}\left[\tilde{Z}_{t}\varphi(X_{t})|\mathcal{Y}_{t}\right] = \int_{\mathbb{R}^{d}}\varphi(x)\sigma(t,x)dx,\ \varphi\text{ is a test function},\ t\in[0,T],
\end{equation}
and the (normalized) conditional expectation $E\left[\varphi(X_{t})|\mathcal{Y}_{t}\right]$ can be calculated by
\begin{equation}
	E\left[\varphi(X_{t})|\mathcal{Y}_{t}\right] = \frac{\tilde{E}\left[\tilde{Z}_{t}\varphi(X_{t})|\mathcal{Y}_{t}\right]}{\tilde{E}\left[\tilde{Z}_{t}|\mathcal{Y}_{t}\right]} = \frac{\int_{\mathbb{R}^{d}}\varphi(x)\sigma(t,x)dx}{\int_{\mathbb{R}^{d}}\sigma(t,x)dx}. 
	\label{eq:3a}
\end{equation}
Because the solution of \eqref{eq:2} does not have a closed form for general nonlinear filtering systems, efficient numerical methods must be proposed, so that we can get a good approximation to the conditional expectation $E[\varphi(X_{t})|\mathcal{Y}_{t}]$ through the equation \eqref{eq:3a}. 

At the beginning of this century, the third author and his collaborator proposed a two-stage algorithm to numerically solve the DMZ equation \eqref{eq:2} in a memoryless and real-time manner, which is often referred to as Yau-Yau algorithm. Here, we would like to briefly introduce the basic idea and main procedure of this algorithm.

Firstly, if we consider the exponential transformation 
\begin{equation}
	w(t,x) := \exp\left(-h^{\top}(x)Y_{t}\right)\sigma(t,x),\ t\in[0,T],
	\label{eq:14}
\end{equation}
then the function $w(t,x)$ satisfies the robust DMZ equation
\begin{equation}
	\frac{\partial w}{\partial t} = \frac{1}{2}\sum_{i,j=1}^{d}a^{ij}(x)\frac{\partial^{2}w}{\partial x_{i}\partial x_{j}} + \sum_{i=1}^{d}F_{i}(t,x)\frac{\partial w}{\partial x_{i}} + J(t,x)w(t,x),
	\label{eq:4a}
\end{equation}
where the stochastic differential terms in the original DMZ equation \eqref{eq:2} are eliminated and
\begin{equation*}
	F_{i}(t,x) = \sum_{j=1}^{d}\biggl(\frac{\partial a^{ij}}{\partial x_{j}} + a^{ij}\sum_{k=1}^{d}Y_{t}^{k}\frac{\partial h_{k}}{\partial x_{j}}\biggr) - f_{i}(x),\quad i = 1,\cdots,d,
\end{equation*}
\begin{equation} 
	\begin{aligned}
		J(t,x) =& \frac{1}{2}\sum_{i,j=1}^{d}\frac{\partial^{2}a^{ij}}{\partial x_{i}\partial x_{j}} + \sum_{i,j,k=1}^{d}Y_{t}^{k}\frac{\partial h_{k}}{\partial x_{j}}\frac{\partial a^{ij}}{\partial x_{i}} \\
		&+ \frac{1}{2}\sum_{i,j=1}^{d}a^{ij}\biggl(\sum_{k=1}^{d}Y_{t}^{k}\frac{\partial^{2}h_{k}}{\partial x_{i}\partial x_{j}} + \sum_{k=1}^{d}\sum_{l=1}^{d}Y_{t}^{k}Y_{t}^{l}\frac{\partial h_{k}}{\partial x_{i}}\frac{\partial h_{l}}{\partial x_{j}}\biggr)\\
		&-\sum_{i=1}^{d}\frac{\partial f_{i}}{\partial x_{i}} - \sum_{i,j=1}^{d}Y_{t}^{j}\frac{\partial h_{j}}{\partial x_{i}}f_{i}(x) - \frac{1}{2}|h|^{2}
	\end{aligned}
	\label{eq:16}
\end{equation}
are stochastic functions that depend on the specific value of observation $Y_{t}$ at time $t$.

Instead of solving equation \eqref{eq:2} directly, we will mainly focus on the robust DMZ equation \eqref{eq:4a}, especially the corresponding initial-boundary value (IBV) problems in a closed ball $B_{R}:=\{x\in\mathbb{R}^{d}:|x|\leq R\}$, with a given radius $R > 0$. 
\begin{equation}
	\left\{\begin{aligned} 
		&\frac{\partial u_{R}}{\partial t} = \frac{1}{2}\sum_{i,j=1}^{d}a^{ij}(x)\frac{\partial^{2}u_{R}}{\partial x_{i}\partial x_{j}} + \sum_{i=1}^{d}F_{i}(t,x)\frac{\partial u_{R}}{\partial x_{i}} + J(t,x)u_{R}(t,x),\ t\in [0,T],\\
		&u_{R}(0,x) = \sigma_{0}(x)\cdot\mathscr{S}_{R}(x),\ x\in B_{R},\\
		&u_{R}(t,x) = 0,\ (t,x)\in[\tau_{k-1},\tau_{k}]\times\partial B_{R}.
	\end{aligned}\right.
	\label{eq:4}
\end{equation}
where $\mathscr{S}_{R}(x)$ is a $C^{\infty}$ function supported in $B_{R}$ and satisfies
\begin{equation}
	\mathscr{S}_{R}(x) = \left\{\begin{aligned}
		&1,\quad |x|\leq R-\frac{1}{R}\\
		&0,\quad |x|\geq R
	\end{aligned}\right.
	\label{eq:19}
\end{equation}
and $0\leq \mathscr{S}_{R}(x)\leq 1$ for $R-\frac{1}{R}\leq |x|\leq R$, so that the initial value is compatible with the boundary condition in \eqref{eq:4}. And from now on, we would like to drop the subscript, $R$, in the notation $u_{R}(t,x)$ for the simplicity of notations, and use $u(t,x)$ to denote the solution to the IBV problem \eqref{eq:19}.

Let $0=\tau_{0}<\tau_{1}<\cdots<\tau_{K} = T$ be a uniform partition of the time interval $[0,T]$, with $\tau_{k}-\tau_{k-1} = \delta=\frac{T}{K}$, $k=1,\cdots, K$. On each time interval $[\tau_{k-1},\tau_{k}]$, consider the IBV problem of the following parabolic equation
\begin{equation}
	\left\{\begin{aligned} 
		&\frac{\partial u_{k}}{\partial t} = \frac{1}{2}\sum_{i,j=1}^{d}a^{ij}(x)\frac{\partial^{2}u_{k}}{\partial x_{i}\partial x_{j}} + \sum_{i=1}^{d}F_{i}(\tau_{k-1},x)\frac{\partial u_{k}}{\partial x_{i}}+ J(\tau_{k-1},x)u_{k}(t,x),\\
		&\qquad\qquad\qquad\qquad\qquad\qquad\qquad\qquad\qquad\qquad(t,x)\in(\tau_{k-1},\tau_{k}]\times B_{R}, \\
		&u_{k}(\tau_{k-1},x) = u_{k-1}(\tau_{k-1},x),\ x\in B_{R},\\
		&u_{k}(t,x) = 0,\ (t,x)\in[\tau_{k-1},\tau_{k}]\times \partial B_{R},
	\end{aligned}\right.
	\label{eq:5}
\end{equation} 
with the value of coefficients $F(t,x)$ and $J(t,x)$ frozen at the left point $t = \tau_{k-1}$ and initial value $u_{0}(\tau_{0},x) :=\sigma_{0}(x)$. 

With another exponential transformation given by
\begin{equation}
	\tilde{u}_{k}(t,x) = \exp\left(h^{\top}(x)Y_{\tau_{k-1}}\right)u_{k}(t,x),\ t\in [\tau_{k-1},\tau_{k}],
	\label{eq:17}
\end{equation}
the newly-constructed function $\tilde{u}_{k}$ satisfies
\begin{equation}
	\left\{\begin{aligned}
		&\frac{\partial \tilde{u}_{k}}{\partial t} = \frac{1}{2}\sum_{i,j=1}^{d}\frac{\partial^{2}}{\partial x_{i}\partial x_{j}}(a^{ij}\tilde{u}_{k}(t,x)) \\
		&\quad- \sum_{i=1}^{d}\frac{\partial}{\partial x_{i}}(f_{i}\tilde{u}_{k}(t,x)) - \frac{1}{2}|h|^{2}\tilde{u}_{k}(t,x),\ (t,x)\in(\tau_{k-1},\tau_{k}]\times B_{R},\\
		&\tilde{u}_{k}(\tau_{k-1},x) = \exp\biggl(h^{\top}(x)Y_{\tau_{k-1}}\biggr)u_{k}(\tau_{k-1},x),\ x\in B_{R}\\
		&\tilde{u}_{k}(t,x) = 0,\ (t,x)\in[\tau_{k-1},\tau_{k}]\times \partial B_{R}
	\end{aligned}\right.
	\label{eq:6}
\end{equation} 
After the two exponential transformations \eqref{eq:14} and \eqref{eq:17}, the function we would like to use to approximate the unnormalized conditional probability density function $\sigma(\tau_{k},x)$ at time $t=\tau_{k}$ is given by
\begin{equation}
	\sigma(\tau_{k},x)\approx \exp\biggl(h^{\top}(x)(Y_{\tau_{k}}-Y_{\tau_{k-1}})\biggr)\tilde{u}_{k}(\tau_{k},x) = \tilde{u}_{k+1}(\tau_{k},x),\  k = 1,\cdots,K.
	\label{eq:8}
\end{equation} 
and the value for approximating the conditional expectation $E[\varphi(X_{t})|\mathcal{Y}_{t}]$ is given by
\begin{equation}
	E[\varphi(X_{t})|\mathcal{Y}_{t}] \approx \frac{\int_{B_{R}}\varphi(x)\tilde{u}_{k+1}(\tau_{k},x)dx}{\int_{B_{R}}\tilde{u}_{k+1}(\tau_{k},x)dx},\ k=1,\cdots,K.
	\label{eq:21}
\end{equation}

The main idea of the Yau-Yau algorithm is that the problem of solving the DMZ equation satisfied by the unnormalized probability density function $\sigma(t,x)$ can be separated into two parts. The computationally expensive part of solving the (IBV) problems of parabolic equation \eqref{eq:6} can be done off-line, because it is a deterministic Kolmogorov-type PDE which is independent of observations, and at least,  the corresponding semi-group, $\{S_{t}:t\in[0,T]\}$, can be analyzed and approximated off-line. When new observation comes, the remaining task is only about calculating exponential transformations and numerical integrals. The framework of Yau-Yau algorithm is shown in Algorithm \ref{alg:1}.
\begin{algorithm}[htb]
	\caption{The Two-Stage Framework of Yau-Yau Algorithm} 
	\begin{algorithmic}[1]
		\STATE \textbf{Initialization:}  Input the terminal time $T$, the radius $R$ of closed ball $B_{R}$, the number of time-discretization steps $K$, the initial distribution of state process $\sigma_{0}(x)$, the test function $\varphi(x)$, and the initial observation $Y_{0} = 0$. Let $\delta = \frac{T}{K}$ be the time-discretization step size and $\{0 = \tau_{0} < \tau_{1}<\cdots<\tau_{K} = T\}$ be a uniform partition of $[0,T]$ with $\tau_{k} -\tau_{k-1} = \delta$. Initialize $\tilde{u}_{1}(0,x) = \sigma_{0}(x)$.
		\STATE \textbf{Off-Line Algorithm:} Solve the IBV problem of Kolmogorov-type partial differential equation \eqref{eq:6} in closed ball $B_{R}$, and determine or approximate the corresponding semi-group $\{S_{t}:t\in[0,T]\}$. 
		\STATE \textbf{On-Line Algorithm:}  
		\FOR{$k = 1$ to $K$} 
		\STATE Obtain $\tilde{u}_{k}(\tau_{k},x)$ from the Off-Line Algorithm
		\begin{equation*}
			\tilde{u}_{k}(\tau_{k},x) = S_{\tau_{k} - \tau_{k-1}}\tilde{u}_{k}(\tau_{k-1},x).
		\end{equation*} 
		\STATE Renew the initial value of the partial differential equation satisfied by $\tilde{u}_{k+1}(x,t)$,
		\begin{equation*}
			\tilde{u}_{k+1}(\tau_{k},x) = \exp\left[h^{\top}(x)(Y_{\tau_{k}} - Y_{\tau_{k-1}})\right]\tilde{u}_{k}(\tau_{k},x).
		\end{equation*}
		\STATE Compute the approximated conditional expectation:
		\begin{equation*}
			\frac{\int_{B_{R}}\varphi(x)\tilde{u}_{k+1}(\tau_{k},x)dx}{\int_{B_{R}}\tilde{u}_{k+1}(\tau_{k},x)dx}.
		\end{equation*}
		\ENDFOR
	\end{algorithmic}
	\label{alg:1}
\end{algorithm}

In the next section, we will give a mathematically rigorous interpretation of the approximation results \eqref{eq:8} and \eqref{eq:21} from a probabilistic perspective. In particular, we only need assumptions on the test function and the coefficients of the filtering systems to derive the convergence result. These assumptions are also easy to verify off-line before the observations come, and therefore, this convergence analysis will provide a guidance for practitioners to determine the parameters in the implementations of Yau-Yau algorithm for practical use.

\section{Main Results}
\label{sec:3}
In this section, we would like to state the main result in this paper and also provide a sketch of the proof. 

Firstly, besides the smoothness and regularity requirements which guarantee the existence of conditional expectation and the existence of the solution to the DMZ equation, let us further introduce four particular assumptions on the coefficients of the system, the initial distribution and the test function.

For the state equation in the filtering system \eqref{eq:1}, the drift term $f:\mathbb{R}^{d}\rightarrow\mathbb{R}^{d}$ is assumed to be Lipschitz, and the the diffusion term $g:\mathbb{R}^{d}\rightarrow\mathbb{R}^{d\times d}$, together with $a(x) = g(x)g(x)^{\top}$, is assumed to have bounded partial derivatives up to second order, i.e.,

\begin{flalign*}
	&\text{\textbf{(A1)}: }\begin{aligned} &\exists\ L>0,\ s.t.\ |f(x)-f(y)|\leq L|x-y|,\\
		&|a(x)|\leq L,\ \biggl|\frac{\partial a^{ij}(x)}{\partial x_{k}}\biggr|\leq L,\ \biggl|\frac{\partial^{2}a^{ij}(x)}{\partial x_{k}\partial x_{l}}\biggr|\leq L,\ \forall \ x,y\in\mathbb{R}^{d},\ i,j,k,l=1,\cdots,d. 
	\end{aligned}
\end{flalign*}

Assumption $\textbf{(A1)}$ guarantees that the state equation of \eqref{eq:1} has a strong solution in $[0,T]$, and especially, the state equation for linear filter satisfies this assumption.   

Also, in order to conduct energy estimations for the (stochastic) partial differential equations, we would like to assume that the diffusion term in the state equation is nondegenerate, in the sense that
\begin{flalign*}
	&\text{\textbf{(A2):}}\quad \begin{aligned} 
		&\text{For each $x\in \mathbb{R}^{d}$, there exists a continuous function $\lambda(x) > 0$},\\ &s.t.\ \sum_{i,j=1}^{d}a^{ij}(x)\zeta_{i}\zeta_{j}\geq \lambda(x)|\zeta|^{2},\ \forall\ \zeta = (\zeta_{1},\cdots,\zeta_{d})^{\top}\in\mathbb{R}^{d}.
	\end{aligned} &
\end{flalign*}

For the initial distribution $\sigma_{0}$, we would like to assume that it is smooth enough and possesses finite high-order moments:
\begin{flalign*}
	&\text{\textbf{(A3):}}\ \int_{\mathbb{R}^{d}}|x|^{2n}\sigma_{0}(x)dx < \infty,\ \forall\ n\in\mathbb{N}.&
\end{flalign*} 

Assumption \textbf{(A3)} is satisfied by commonly-used light-tailed distributions such as Gaussian distributions. In fact, in the following convergence analysis, we only require Assumption \textbf{(A3)} to hold for \textit{sufficiently} large $n\geq 1$, rather than for all $n\in\mathbb{N}$.

Finally, it is assumed that the test function $\varphi:\mathbb{R}^{d}\rightarrow\mathbb{R}$ is at most polynomial growth:
\begin{flalign*}
	&\text{\textbf{(A4): }}\exists \ L>0,\ m\in\mathbb{N},\ s.t.\ |\varphi(x)|\leq L(1+|x|^{2m}),\quad \forall \ x\in \mathbb{R}^{d}, &
\end{flalign*}
which is satisfied by most of the commonly-used test functions, such as those correspond to the conditional mean and covariance matrix.

Based on the above assumptions \textbf{(A1)} to \textbf{(A4)}, the main result in this paper is stated as follows:
\begin{theorem}
	Fix a terminal time $T > 0$ and the filtering system \eqref{eq:1} with smooth coefficients. If Assumptions \textbf{(A1)} to \textbf{(A4)} hold, then for every $\epsilon > 0$, there exists $R > 0$, and $\delta > 0$, such that we can conduct the Yau-Yau algorithm in the closed ball $B_{R} = \{x\in \mathbb{R}^{d}:|x|\leq R\}$ and the uniform partition of $[0,T]$: $0=\tau_{0}<\tau_{1}<\cdots <\tau_{K} = T$ with $\delta = \tau_{k}-\tau_{k-1}$ for $k=1,\cdots,K$, and the numerical solution $\{\tilde{u}_{k+1}(\tau_{k},x):k=1,\cdots,K\}$ approximates the exact solution of the filtering problems at each time step $t = \tau_{k}$ well in the sense of mathematical expectation, i.e.,  
	\begin{equation}
		E\biggl|E[\varphi(X_{\tau_{k}})|\mathcal{Y}_{\tau_{k}}] - \frac{\int_{B_{R}}\varphi(x)\tilde{u}_{k+1}(\tau_{k},x)dx}{\int_{B_{R}}\tilde{u}_{k+1}(\tau_{k},x)dx}\biggr| < \epsilon,\ \forall\ 1\leq k\leq K. 
	\end{equation} 
	\label{thm:4}
\end{theorem}

Here in this section, we provide a sketch of the proof of Theorem \ref{thm:4}, in which the main idea of the proof is illustrated. The detailed proofs of those key estimations here will be given in order in the next four sections. 
\begin{proof}[A Sketch of the Proof of Theorem \ref{thm:4}]
	Let $\{\tilde{Z}_{t}:0\leq t\leq T\}$ be the Radon derivative $Z_{t} = \left.\frac{dP}{d\tilde{P}_{t}}\right|_{\mathcal{F}_{t}}$ defined in \eqref{eq:4b}. And therefore, for every integrable, $\mathcal{F}_{t}$-measurable random variable $U$, we have
	\begin{equation}
		E[U] = \tilde{E}[\tilde{Z}_{t}U].
	\end{equation} 
	According to the properties of conditional expectations, the expectation of the approximation error of Yau-Yau algorithm can be estimated as follows:
	\begin{align*}
		E&\biggl|E[\varphi(X_{\tau_{k}})|\mathcal{Y}_{\tau_{k}}] - \frac{\int_{B_{R}}\varphi(x)\tilde{u}_{k+1}(\tau_{k},x)dx}{\int_{B_{R}}\tilde{u}_{k+1}(\tau_{k},x)dx}\biggr|\\
		&=\tilde{E}\biggl[\tilde{Z}_{\tau_{k}}\biggl|E[\varphi(X_{\tau_{k}})|\mathcal{Y}_{\tau_{k}}] - \frac{\int_{B_{R}}\varphi(x)\tilde{u}_{k+1}(\tau_{k},x)dx}{\int_{B_{R}}\tilde{u}_{k+1}(\tau_{k},x)dx}\biggr|\biggr]\\
		&= \tilde{E}\biggl[\tilde{Z}_{\tau_{k}}\biggl|\frac{\tilde{E}[\varphi(X_{\tau_{k}})\tilde{Z}_{\tau_{k}}|\mathcal{Y}_{\tau_{k}}]}{\tilde{E}[\tilde{Z}_{\tau_{k}}|\mathcal{Y}_{\tau_{k}}]} - \frac{\int_{B_{R}}\varphi(x)\tilde{u}_{k+1}(\tau_{k},x)dx}{\int_{B_{R}}\tilde{u}_{k+1}(\tau_{k},x)dx}\biggr|\biggr]\\
		&= \tilde{E}\biggl[\tilde{E}[\tilde{Z}_{\tau_{k}}|\mathcal{Y}_{\tau_{k}}]\biggl|\frac{\tilde{E}[\varphi(X_{\tau_{k}})\tilde{Z}_{\tau_{k}}|\mathcal{Y}_{\tau_{k}}]}{\tilde{E}[\tilde{Z}_{\tau_{k}}|\mathcal{Y}_{\tau_{k}}]} - \frac{\int_{B_{R}}\varphi(x)\tilde{u}_{k+1}(\tau_{k},x)dx}{\int_{B_{R}}\tilde{u}_{k+1}(\tau_{k},x)dx}\biggr|\biggr]\\
		&\leq  \tilde{E}\biggl[\tilde{E}[\tilde{Z}_{\tau_{k}}|\mathcal{Y}_{\tau_{k}}]\biggl|\frac{\tilde{E}[\varphi(X_{\tau_{k}})\tilde{Z}_{\tau_{k}}|\mathcal{Y}_{\tau_{k}}]}{\tilde{E}[\tilde{Z}_{\tau_{k}}|\mathcal{Y}_{\tau_{k}}]} - \frac{\int_{B_{R}}\varphi(x)\tilde{u}_{k+1}(\tau_{k},x)dx}{\tilde{E}[\tilde{Z}_{\tau_{k}}|\mathcal{Y}_{\tau_{k}}]}\biggr|\biggr]\\
		&\quad + \tilde{E}\biggl[\tilde{E}[\tilde{Z}_{\tau_{k}}|\mathcal{Y}_{\tau_{k}}]\biggl|\frac{\int_{B_{R}}\varphi(x)\tilde{u}_{k+1}(\tau_{k},x)dx}{\tilde{E}[\tilde{Z}_{\tau_{k}}|\mathcal{Y}_{\tau_{k}}]} - \frac{\int_{B_{R}}\varphi(x)\tilde{u}_{k+1}(\tau_{k},x)dx}{\int_{B_{R}}\tilde{u}_{k+1}(\tau_{k},x)dx}\biggr|\biggr]\\
		&\leq  \tilde{E}\biggl[\biggl|\tilde{E}[\varphi(X_{\tau_{k}})\tilde{Z}_{\tau_{k}}|\mathcal{Y}_{\tau_{k}}]-\int_{B_{R}}\varphi(x)\tilde{u}_{k+1}(\tau_{k},x)dx\biggr|\biggr]\\
		&\quad + \tilde{E}\biggl[\frac{\int_{B_{R}}|\varphi(x)|\tilde{u}_{k+1}(\tau_{k},x)dx}{\int_{B_{R}}\tilde{u}_{k+1}(\tau_{k},x)dx}\biggl|\tilde{E}[\tilde{Z}_{\tau_{k}}|\mathcal{Y}_{\tau_{k}}] - \int_{B_{R}}\tilde{u}_{k+1}(\tau_{k},x)dx\biggr|\biggr]\\
		&= \tilde{E}\biggl[\biggl|\int_{\mathbb{R}^{d}}\varphi(x)\sigma(\tau_{k},x)dx - \int_{B_{R}}\varphi(x)\tilde{u}_{k+1}(\tau_{k},x)dx\biggr|\biggr]\\
		&\quad + \tilde{E}\biggl[\frac{\int_{B_{R}}|\varphi(x)|\tilde{u}_{k+1}(\tau_{k},x)dx}{\int_{B_{R}}\tilde{u}_{k+1}(\tau_{k},x)dx}\biggl|\int_{\mathbb{R}^{d}}\sigma(\tau_{k},x)dx - \int_{B_{R}}\tilde{u}_{k+1}(\tau_{k},x)dx\biggr|\biggr]\\
		&\triangleq I_{1} + I_{2}.
	\end{align*}
	where we use the fact that $\tilde{u}_{k+1}(\tau_{k},x)$ is $\mathcal{Y}_{\tau_{k}}$-measurable and for integrable, $\mathcal{Y}_{\tau_{k}}$-measurable random variable $V$,  
	\begin{equation}
		\tilde{E}[\tilde{Z}_{\tau_{k}}V] = \tilde{E}[\tilde{E}[\tilde{Z}_{\tau_{k}}|\mathcal{Y}_{\tau_{k}}]V].
	\end{equation}
	
	Therefore, the remaining task for us is to estimate the two error terms $I_{1}$ and $I_{2}$, and to show that $I_{1}$ and $I_{2}$ can be arbitrarily small with sufficiently large $R> 0$ and sufficiently small $\delta > 0$.
	
	Firstly, for the estimation of 
	\begin{equation}
		I_{1} = \tilde{E}\biggl[\biggl|\int_{\mathbb{R}^{d}}\varphi(x)\sigma(\tau_{k},x)dx - \int_{B_{R}}\varphi(x)\tilde{u}_{k+1}(\tau_{k},x)dx\biggr|\biggr],
	\end{equation}
	we would like to utilize an intermediate function $\sigma_{R}(t,x)$, $(t,x)\in[0,T]\times B_{R}$, which is the solution of IBV problem of the DMZ equation \eqref{eq:2} and will be introduced in \eqref{eq:13a} in Section \ref{sec:4}. And we have
	\begin{align}
		I_{1}\leq &\tilde{E}\int_{|x|\geq R}|\varphi(x)|\sigma(\tau_{k},x)dx + \tilde{E}\biggl[\biggl|\int_{B_{R}}\varphi(x)\sigma(\tau_{k},x)dx - \int_{B_{R}}\varphi(x)\sigma_{R}(\tau_{k},x)dx\biggr|\biggr] \notag\\
		&+ \tilde{E}\biggl[\biggl|\int_{B_{R}}\varphi(x)\sigma_{R}(\tau_{k},x)dx - \int_{B_{R}}\varphi(x)\tilde{u}_{k+1}(\tau_{k},x)dx\biggr|\biggr]\notag\\
		\leq &\tilde{E}\int_{|x|\geq R}|\varphi(x)|\sigma(\tau_{k},x)dx + \tilde{E}\int_{B_{R}}|\varphi(x)|\cdot|\sigma(\tau_{k},x) - \sigma_{R}(\tau_{k},x)|dx \notag\\
		&+ \tilde{E}\int_{B_{R}}|\varphi(x)|\cdot|\sigma_{R}(\tau_{k},x)-\tilde{u}_{k+1}(\tau_{k},x)|dx\notag\\
		\leq &\tilde{E}\int_{|x|\geq R}|\varphi(x)|\sigma(\tau_{k},x)dx + L(1+R^{2m})\tilde{E}\int_{B_{R}}|\sigma(\tau_{k},x) - \sigma_{R}(\tau_{k},x)|dx \notag\\
		&+ L(1+R^{2m})\tilde{E}\int_{B_{R}}|\sigma_{R}(\tau_{k},x)-\tilde{u}_{k+1}(\tau_{k},x)|dx.
		\label{eq:29}
	\end{align}
	
	For the estimation of 
	\begin{equation}
		I_{2} = \tilde{E}\biggl[\frac{\int_{B_{R}}|\varphi(x)|\tilde{u}_{k+1}(\tau_{k},x)dx}{\int_{B_{R}}\tilde{u}_{k+1}(\tau_{k},x)dx}\biggl|\int_{\mathbb{R}^{d}}\sigma(\tau_{k},x)dx - \int_{B_{R}}\tilde{u}_{k+1}(\tau_{k},x)dx\biggr|\biggr],
	\end{equation}
	since in the closed ball $B_{R}$, $|\varphi(x)|\leq L(1+R^{2m})$,
	\begin{equation}
		\frac{\int_{B_{R}}|\varphi(x)|\tilde{u}_{k+1}(\tau_{k},x)dx}{\int_{B_{R}}\tilde{u}_{k+1}(\tau_{k},x)dx}\leq L(1+R^{2m})\frac{\int_{B_{R}}\tilde{u}_{k+1}(\tau_{k},x)dx}{\int_{B_{R}}\tilde{u}_{k+1}(\tau_{k},x)dx} = L(1+R^{m}).
	\end{equation}
	and thus,
	\begin{equation}
		\begin{aligned} 
			I_{2}&\leq L(1+R^{2m})\tilde{E}\biggl[\biggl|\int_{\mathbb{R}^{d}}\sigma(\tau_{k},x)dx - \int_{B_{R}}\tilde{u}_{k+1}(\tau_{k},x)dx\biggr|\biggr]\\
			&\leq L(1+R^{2m})\biggl(\tilde{E}\int_{|x|\geq R}\sigma(\tau_{k},x)dx + \tilde{E}\biggl[\biggl|\int_{B_{R}}\sigma(\tau_{k},x)dx - \int_{B_{R}}\sigma_{R}(\tau_{k},x)dx\biggr|\biggr]\biggr)\\
			&\quad+ L(1+R^{2m})\tilde{E}\biggl[\biggl|\int_{B_{R}}\sigma_{R}(\tau_{k},x)dx - \int_{B_{R}}\tilde{u}_{k+1}(\tau_{k},x)dx\biggr|\biggr]\\
			&\leq L(1+R^{2m})\biggl(\tilde{E}\int_{|x|\geq R}\sigma(\tau_{k},x)dx + \tilde{E}\int_{B_{R}}|\sigma(\tau_{k},x)-\sigma_{R}(\tau_{k},x)|dx\biggr)\\
			&\quad+ L(1+R^{2m})\tilde{E}\int_{B_{R}}|\sigma_{R}(\tau_{k},x)-\tilde{u}_{k+1}(\tau_{k},x)|dx.
		\end{aligned} 
		\label{eq:32}
	\end{equation}
	Combining \eqref{eq:29} and \eqref{eq:32}, we have
	\begin{equation}
		\begin{aligned}
			E&\biggl|E[\varphi(X_{\tau_{k}})|\mathcal{Y}_{\tau_{k}}] - \frac{\int_{B_{R}}\varphi(x)\tilde{u}_{k+1}(\tau_{k},x)dx}{\int_{B_{R}}\tilde{u}_{k+1}(\tau_{k},x)dx}\biggr|\leq I_{1} + I_{2}\\
			&\leq \tilde{E}\int_{|x|\geq R}|\varphi(x)|\sigma(\tau_{k},x)dx + L(1+R^{m})\tilde{E}\int_{|x|\geq R}\sigma(\tau_{k},x)dx\\ &\quad+2L(1+R^{2m})\tilde{E}\int_{B_{R}}|\sigma(\tau_{k},x) - \sigma_{R}(\tau_{k},x)|dx \\
			&\quad+ 2L(1+R^{2m})\tilde{E}\int_{B_{R}}|\sigma_{R}(\tau_{k},x)-\tilde{u}_{k+1}(\tau_{k},x)|dx.
		\end{aligned}
		\label{eq:33}
	\end{equation}
	
	According to Theorem \ref{thm:1b} in Section \ref{sec:4}, for every $n\in\mathbb{N}$, there exists $C_{1}>0$, which depends on $d$, $m$, $n$, $L$, $T$, such that 
	\begin{equation}
		\begin{aligned} 
			&\tilde{E}\int_{|x|\geq R}\sigma(\tau_{k},x)dx \leq \frac{C_{1}}{1+R^{2n}}\int_{\mathbb{R}^{d}}|x|^{2n}\sigma_{0}(x)dx\\
			&\tilde{E}\int_{|x|\geq R}|\varphi(x)|\sigma(\tau_{k},x)dx\leq \frac{C_{1}}{1+R^{2n}}\int_{\mathbb{R}^{d}}|x|^{2(m+n)}\sigma_{0}(x)dx
		\end{aligned} 
		\label{eq:34}
	\end{equation}
	
	Therefore, for every $\epsilon > 0$, with Assumption \textbf{(A3)} for the initial distribution $\sigma_{0}$, as long as we take $n > m$, there exists $R_{1} > 0$, such that
	\begin{equation}
		\begin{aligned}
			\tilde{E}&\int_{|x|\geq R_{1}}|\varphi(x)|\sigma(\tau_{k},x)dx + L(1+R_{1}^{m})\tilde{E}\int_{|x|\geq R_{1}}\sigma(\tau_{k},x)dx\\
			&\leq \frac{C_{1}(1+R_{1}^{2m})}{1+R_{1}^{2n}}\int_{\mathbb{R}^{d}}|x|^{2n}\sigma_{0}(x)dx + \frac{C_{1}}{1+R_{1}^{2n}}\int_{\mathbb{R}^{d}}|x|^{2(m+n)}\sigma_{0}(x)dx < \frac{\epsilon}{3}
		\end{aligned}
		\label{eq:35}
	\end{equation}
	
	According to Theorem \ref{thm:2} in Section \ref{sec:5}, there exists $C_{2} > 0$, which depends on $d,n,L,T$, such that
	\begin{equation}
		\tilde{E}\int_{B_{R}}|\sigma(\tau_{k},x) - \sigma_{R}(\tau_{k},x)|dx\leq \frac{C_{2}}{1+R^{2n}}.
		\label{eq:36}
	\end{equation}
	
	Therefore, as long as $n>m$, there exists $R_{2}>0$, such that
	\begin{equation}
		\begin{aligned} 
			2L(1+R_{2}^{2m})\tilde{E}\int_{B_{R_{2}}}|\sigma(\tau_{k},x) - \sigma_{R_{2}}(\tau_{k},x)|dx\leq\frac{2C_{2}L(1+R_{2}^{2m})}{1+R_{2}^{2n}}<\frac{\epsilon}{3} 
		\end{aligned}
		\label{eq:37}
	\end{equation}
	
	Let us choose $R = \max\{R_{1},R_{2}\}$, and for this particular $R$, according to Theorem \ref{thm:1a} in Section \ref{sec:7}, there exists a time step $\delta > 0$, such that
	\begin{equation}
		\tilde{E}\int_{B_{R}}|\sigma_{R}(\tau_{k},x) - \tilde{u}_{k+1}(\tau_{k},x)|dx <\frac{\epsilon}{6L(1+R^{2m})}. 
		\label{eq:38}
	\end{equation}
	and thus,
	\begin{equation}
		2L(1+R^{2m})\tilde{E}\int_{B_{R}}|\sigma_{R}(\tau_{k},x)-\tilde{u}_{k+1}(\tau_{k},x)|dx\leq \frac{2L(1+R^{2m})\epsilon}{6L(1+R^{2m})}<\frac{\epsilon}{3}.
		\label{eq:39}
	\end{equation}
	
	Take $\eqref{eq:35}$, $\eqref{eq:37}$ and $\eqref{eq:39}$ back to \eqref{eq:33}, and we obtain the desired result, that is, we have found $R > 0$ and $\delta > 0$, such that 
	\begin{equation}
		E\biggl|E[\varphi(X_{\tau_{k}})|\mathcal{Y}_{\tau_{k}}] - \frac{\int_{B_{R}}\varphi(x)\tilde{u}_{k+1}(\tau_{k},x)dx}{\int_{B_{R}}\tilde{u}_{k+1}(\tau_{k},x)dx}\biggr| < \epsilon.
	\end{equation}
\end{proof}
\section{Estimation of the density outside the ball \texorpdfstring{$B_{R}$}{BR}}
\label{sec:4}
In this section, we will provide an estimation of the value of the unnormalized conditional probability density $\sigma(t,x)$ outside a ball $B_{R}\subset\mathbb{R}^{d}$, with $R\gg1$ large enough. 

Especially, we will show that almost all the density of $\sigma(t,x)$ is contained in the closed ball $B_{R}$, and the estimations \eqref{eq:34} in the proof of Theorem \ref{thm:4} in Section \ref{sec:3} holds with Assumptions \textbf{(A1)} to \textbf{(A4)}.
\begin{theorem}
	With Assumptions \textbf{(A1)} to \textbf{(A4)}, there exists a constant $C > 0$ which only depends on $T$, $L$, $d$, $m$ and $n$, such that
	\begin{equation}
		\sup\limits_{0\leq t\leq T}\tilde{E}\int_{|x|\geq R}\sigma(t,x)dx\leq \frac{C}{1+R^{2n}}\int_{\mathbb{R}^{d}}|x|^{2n}\sigma_{0}(x)dx,
	\end{equation}
	and
	\begin{equation}
		\sup\limits_{0\leq t\leq T}\tilde{E}\int_{|x|\geq R}|\varphi(x)|\sigma(t,x)dx\leq \frac{C}{1+R^{2n}}\int_{\mathbb{R}^{d}}|x|^{2(m+n)}\sigma_{0}(x)dx
	\end{equation}
	holds for all $R > 0$. 
	\label{thm:1b}
\end{theorem}

\begin{proof}[Proof of Theorem \ref{thm:1b}]
	We first consider the following IBV problem on the ball $B_{R}$:
	\begin{equation}
		\left\{\begin{aligned}
			&d\sigma_{R}(t,x) = \mathcal{L}\sigma_{R}(t,x)dt + \sum_{j=1}^{d}h_{j}(x)\sigma_{R}(t,x)dY_{t}^{j},\ (t,x)\in (0,T]\times B_{R};\\
			&\sigma_{R}(0,x) = \sigma_{0,R}(x)\triangleq\sigma_{0}(x)\cdot \mathscr{S}_{R}(x),\ x\in B_{R};\\
			&\sigma_{R}(t,x) = 0, \ (t,x)\in[0,T]\times \partial B_{R}.
		\end{aligned}\right.
		\label{eq:13a}
	\end{equation} 
	where $\mathscr{S}_{R}(x)$ is the $C^{\infty}$ function defined in \eqref{eq:19}, such that the initial value is compatible with the boundary conditions.
	
	Let $\psi(x) = \log\left(1+|x|^{2n}\right)$ and define
	\begin{equation}
		\Phi(t) = \int_{B_{R}}e^{\psi(x)}\sigma_{R}(t,x)dx.
	\end{equation}
	Then, according to the IBV problem \eqref{eq:13a} satisfied by the function $\sigma_{R}(t,x)$, we have
	\begin{equation}
		\begin{aligned}
			d\Phi(t) = & \biggl[\frac{1}{2}\sum_{i,j=1}^{d}\int_{B_{R}}\frac{\partial^{2}}{\partial x_{i}\partial x_{j}}\left[\left(a^{ij}(x)\right)\sigma_{R}(t,x)\right]e^{\psi(x)}dx \\
			&- \sum_{i=1}^{d}\int_{B_{R}}\frac{\partial}{\partial x_{i}}\left(f_{i}(x)\sigma_{R}(t,x)\right)e^{\psi(x)}dx\biggr]dt\\
			& + \sum_{j=1}^{d}\biggl[\int_{B_{R}}e^{\psi(x)}h_{j}(x)\sigma_{R}(t,x)dx
			\biggr]dY_{t}^{j}\\
			\triangleq& \left[I_{1}(t) - I_{2}(t)\right]dt + \sum_{j=1}^{d}I_{3,j}(t) dY_{t}^{j}.
		\end{aligned}
	\end{equation}
	By the Gauss-Green formula, we have
	\begin{equation}
		\begin{aligned} 
			I_{1}(t) = \frac{1}{2}\sum_{i,j=1}^{d}\biggl[&\int_{B_{R}}e^{\psi(x)}\left(\frac{\partial\psi(x)}{\partial x_{i}}\frac{\partial\psi(x)}{\partial x_{j}} + \frac{\partial^{2}\psi(x)}{\partial x_{i}\partial x_{j}}\right)a^{ij}(x)\sigma_{R}(t,x)dx\\
			&-\int_{B_{R}}\frac{\partial}{\partial x_{j}}\left(e^{\psi(x)}\frac{\partial\psi(x)}{\partial x_{i}}a^{ij}(x)\sigma_{R}(t,x)\right)dx\\
			&+ \int_{B_{R}}\frac{\partial}{\partial x_{i}}\left(e^{\psi(x)}\frac{\partial}{\partial x_{j}}\left[a^{ij}(x)\sigma_{R}(t,x)\right]\right)dx\biggr]\\
			= \frac{1}{2}\sum_{i,j=1}^{d}&\int_{B_{R}}e^{\psi(x)}\left(\frac{\partial\psi(x)}{\partial x_{i}}\frac{\partial\psi(x)}{\partial x_{j}} + \frac{\partial^{2}\psi(x)}{\partial x_{i}\partial x_{j}}\right)a^{ij}(x)\sigma_{R}(t,x)dx\\
			&-\int_{\partial B_{R}}\vec{\mathfrak{M}}_{1}(t,x)\cdot\vec{\mathfrak{n}}dS + \int_{\partial B_{R}}\vec{\mathfrak{M}}_{2}(t,x)\cdot \vec{\mathfrak{n}}dS,
		\end{aligned}
	\end{equation}
	where $\vec{\mathfrak{n}}$ is the unit outward normal vector of $\partial B_{R}$, $dS$ denotes the measure on $\partial B_{R}$, 
	\begin{equation}
		\vec{\mathfrak{M}}_{i}(t,x) = \left(\mathfrak{M}_{i,1}(t,x),\cdots,\mathfrak{M}_{i,d}(t,x)\right),\quad i = 1,2,
	\end{equation}
	and 
	\begin{equation}
		\begin{aligned} 
			&\mathfrak{M}_{1,j}(t,x) = \frac{1}{2}\sum_{i=1}^{d}e^{\psi(x)}\frac{\partial\psi(x)}{\partial x_{i}}a^{ij}(x)\sigma_{R}(t,x),\quad j=1,\cdots,d,\\
			&\mathfrak{M}_{2,i}(t,x) = \frac{1}{2}\sum_{j=1}^{d}e^{\psi(x)}\frac{\partial}{\partial x_{j}}\left[a^{ij}(x)\sigma_{R}(t,x)\right],\quad i=1,\cdots,d.
		\end{aligned}
	\end{equation} 
	Since $\sigma_{R}(t,x)\equiv 0$, $\forall (t,x)\in [0,T]\times\partial B_{R}$, $\mathfrak{M}_{1,j}(t,x)\equiv 0$ on $[0,T]\times\partial B_{R}$ and 
	\begin{equation}
		\int_{\partial B_{R}}\vec{\mathfrak{M}}_{1}(t,x)\cdot\vec{\mathfrak{n}}dS = 0.
	\end{equation} 
	Moreover, we have
	\begin{equation}
		\nabla \sigma_{R} = \left(\frac{\partial \sigma_{R}}{\partial x_{1}},\cdots,\frac{\partial \sigma_{R}}{\partial x_{d}}\right)= -c\ \vec{\mathfrak{n}},\quad \text{on}\ \partial B_{R},
	\end{equation}
	where $c(x)> 0$ is a continuous function on $\partial B_{R}$, because $\sigma_{R}\geq 0$ and $\sigma_{R}|_{\partial B_{R}} \equiv 0$. 
	
	Therefore,
	\begin{equation}
		\begin{aligned} 
			\vec{\mathfrak{M}}_{2}(t,x)\cdot \vec{\mathfrak{n}} = &-\vec{\mathfrak{M}}_{2}(t,x)\cdot \nabla \sigma_{R}\\
			=& -e^{\psi(x)}\frac{1}{2}\sum_{i,j=1}^{d}a^{ij}(x)\frac{\partial \sigma_{R}}{\partial x_{j}}\frac{\partial \sigma_{R}}{\partial x_{i}} \\
			&- e^{\psi(x)}\sigma_{R}\frac{1}{2}\sum_{i,j=1}^{d}\frac{\partial}{\partial x_{j}}a^{ij}(x)\frac{\partial \sigma_{R}}{\partial x_{i}}\\
			= & -e^{\psi(x)}\frac{1}{2}\sum_{i,j=1}^{d}a^{ij}(x)\frac{\partial \sigma_{R}}{\partial x_{j}}\frac{\partial \sigma_{R}}{\partial x_{i}}\leq 0,\quad \text{on}\ \partial B_{R},
		\end{aligned}
	\end{equation}
	where the last inequality holds because $a(x) = g(x)g(x)^{\top}$ is positive semi-definite. Thus,
	\begin{equation}
		I_{1}(t) \leq  \frac{1}{2}\sum_{i,j=1}^{d}\int_{B_{R}}e^{\psi(x)}\left(\frac{\partial\psi(x)}{\partial x_{i}}\frac{\partial\psi(x)}{\partial x_{j}} + \frac{\partial^{2}\psi(x)}{\partial x_{i}\partial x_{j}}\right)a^{ij}(x)\sigma_{R}(t,x)dx.
	\end{equation}
	Similarly, 
	\begin{align}
		&I_{2}(t) = -\sum_{i=1}^{d}\int_{B_{R}}f_{i}(x)\sigma_{R}(t,x)e^{\psi(x)}\frac{\partial\psi(x)}{\partial x_{i}}dx.
	\end{align}
	Therefore,
	\begin{equation}
		d\Phi(t)\leq \left(\int_{B_{R}}\mathfrak{F}(x)e^{\psi(x)}\sigma_{R}(t,x)dx\right)dt + \sum_{j=1}^{d}\left(\int_{B_{R}}h_{j}(x)e^{\psi(x)}u(t,x)dx\right)dY_{t}^{j},
		\label{eq:14a}
	\end{equation}
	where
	\begin{align}
		&\mathfrak{F}(x) = \frac{1}{2}\sum_{i,j=1}^{d}\left(\frac{\partial\psi(x)}{\partial x_{i}}\frac{\partial\psi(x)}{\partial x_{j}} + \frac{\partial^{2}\psi(x)}{\partial x_{i}\partial x_{j}}\right)a^{ij}(x) + \sum_{i=1}^{d}f_{i}(x)\frac{\partial\psi(x)}{\partial x_{i}}\label{eq:40}.
	\end{align}
	
	Since $\psi(x) = \log\left(1+|x|^{2n}\right)$, then
	\begin{equation}
		\frac{\partial\psi}{\partial x_{i}} = \frac{2n|x|^{2n-2}x_{i}}{1+|x|^{2n}},\ \frac{\partial^{2}\psi}{\partial x_{i}\partial x_{j}} = \frac{4nx_{i}x_{j}|x|^{2n-4}\left((n-1)-|x|^{2n}\right)}{\left(1+|x|^{2n}\right)^{2}} + \frac{2n|x|^{2n-2}\mathscr{\delta}_{ij}}{1+|x|^{2n}},
	\end{equation}
	where $\mathcal{\delta}_{ij}$ is the Kronecker's symbol, with $\mathcal{\delta}_{ij} = 1$, if $i=j$, and $\mathcal{\delta}_{ij} = 0$ otherwise.
	
	Notice that
	\begin{equation}
		\left|\frac{\partial\psi}{\partial x_{i}}\right|\leq 2n,\ \left|\frac{\partial^{2}\psi}{\partial x_{i}\partial x_{j}}\right|\leq 4n^{2}+2n.
	\end{equation}
	With the assumption that $|a^{ij}(x)|\leq L$, we have
	\begin{equation}
		|\mathfrak{F}(x)|\leq d^{2}(4n^{2}+n)L + 2n\sum_{i=1}^{d}\frac{|f_{i}(x)x_{i}|\cdot|x|^{2n-2}}{1+|x|^{2n}},\ \forall x\in\mathbb{R}^{d}.
	\end{equation}
	
	Because $f(x)$ is Lipschitz continuous according to (\textbf{A1}). 
	\begin{equation}
		|f(x)|\leq L|x| + |f(0)|,\quad \forall x\in \mathbb{R}^{d}.
	\end{equation}
	Therefore,
	\begin{equation}
		\begin{aligned} 
			|\mathfrak{F}(x)|&\leq d^{2}(4n^{2}+n)L + \frac{2n|x|^{2n-2}}{1+|x|^{2n}}\sum_{i=1}^{d}\frac{|f_{i}(x)|^{2}+|x_{i}|^{2}}{2}\\
			&\leq d^{2}(4n^{2}+n)L + \frac{n|x|^{2n}+ n|x|^{2n-2}|f(x)|^{2}}{1+|x|^{2n}}\\
			&\leq d^{2}(4n^{2}+n)L + n(L^{2}+1) + n|f(0)|^{2} + 2nL|f(0)|,\ \forall x\in\mathbb{R}^{d}.
		\end{aligned} 
	\end{equation}
	Let us denote by $M(n,d,L)$ the above upper bound of $|\mathfrak{F}(x)|$:
	\begin{equation}
		M(n,d,L) := d^{2}(4n^{2}+n)L + n(L^{2}+1) + n|f(0)|^{2} + 2nL|f(0)|,
	\end{equation}
	which is a constant that depends on $n$, $d$ and $L$, but does not depend on $R$. 
	
	Take expectation with respect to the reference probability measure $\tilde{P}$, we obtain
	\begin{equation}
		\frac{d}{dt}\tilde{E}\Phi(t)\leq M(n,d,L)\tilde{E}\Phi(t). 
	\end{equation}
	Here we use the fact that $Y_{t}$ is a Brownian motion with respect to $\tilde{P}$.
	
	According to the Gronwall's inequality, we have
	\begin{equation}
		\begin{aligned} 
			\sup\limits_{0\leq t\leq T}\tilde{E}\int_{B_{R}}\left(1+|x|^{2n}\right)\sigma_{R}(t,x)dx&\leq e^{M(n,d,L)T}\int_{B_{R}}\left(1+|x|^{2n}\right)\sigma_{0,R}(x)dx\\
			&\leq e^{M(n,d,L)T}\int_{B_{R}}\left(1+|x|^{2n}\right)\sigma_{0}(x)dx.
		\end{aligned}
	\end{equation}
	
	Let $R$ tends to infinity, and we have
	\begin{equation}
		\sup\limits_{0\leq t\leq T}\tilde{E}\int_{\mathbb{R}^{d}}\left(1+|x|^{2n}\right)\sigma(t,x)dx\leq e^{M(n,d,L)T}\int_{\mathbb{R}^{d}}\left(1+|x|^{2n}\right)\sigma_{0}(x)dx. 
	\end{equation}
	Therefore,
	\begin{equation}
		\begin{aligned} 
			\sup\limits_{0\leq t\leq T} \tilde{E}\int_{|x|\geq R}\sigma(t,x)dx&\leq \frac{1}{1+R^{2n}}\sup\limits_{0\leq t\leq T}\tilde{E}\int_{|x|\geq R}\left(1+|x|^{2n}\right)\sigma(t,x)dx\\
			&\leq \frac{1}{1+R^{2n}}\sup\limits_{0\leq t\leq T}\tilde{E}\int_{\mathbb{R}^{d}}\left(1+|x|^{2n}\right)\sigma(t,x)dx\\
			&\leq \frac{e^{M(n,d,L)T}}{1+ R^{2n}}\int_{\mathbb{R}^{d}}\left(1+|x|^{2n}\right)\sigma_{0}(x)dx.
		\end{aligned}
	\end{equation}
	Moreover, with condition \textbf{(A4)},
	\begin{equation}
		\begin{aligned} 
			\sup\limits_{0\leq t\leq T} \tilde{E}&\int_{|x|\geq R}|\varphi(x)|\sigma(t,x)dx\leq \sup\limits_{0\leq t\leq T}\tilde{E}\int_{|x|\geq R}\left(1+|x|^{2m}\right)\sigma(t,x)dx\\
			& \leq \frac{1}{1+R^{2n}}\sup\limits_{0\leq t\leq T}\tilde{E}\int_{|x|\geq R}\left(1+|x|^{2n}\right)\left(1+|x|^{2m}\right)\sigma(t,x)dx\\
			&\leq \frac{1}{1+R^{2n}}\sup\limits_{0\leq t\leq T}\tilde{E}\int_{|x|\geq R}2\left(1+|x|^{2(m+n)}\right)\sigma(t,x)dx\\
			&\leq \frac{2}{1+R^{2n}}\sup\limits_{0\leq t\leq T}\tilde{E}\int_{\mathbb{R}^{d}}\left(1+|x|^{2(m+n)}\right)\sigma(t,x)dx\\
			&\leq \frac{2e^{M(n+m,d,L)T}}{1+ R^{2n}}\int_{\mathbb{R}^{d}}\left(1+|x|^{2(m+n)}\right)\sigma_{0}(x)dx.
		\end{aligned}
	\end{equation}
\end{proof}
\section{Approximation of \texorpdfstring{$\sigma(t,x)$}{sigma(t,x)} by the IBV problem in \texorpdfstring{$B_{R}$}{BR}}
\label{sec:5}
With the estimation in Theorem \ref{thm:1b}, because almost all the density of $\sigma(t,x)$ is contained in the closed ball $B_{R}$ for $R$ large enough, it is natural to think about approximating $\sigma(t,x)$ by the solution, $\sigma_{R}(t,x)$, to the corresponding initial-boundary value (IBV) problem \eqref{eq:13a} of DMZ equation in the ball $B_{R}$.

It will be rigorously proved in this section that, for $R$ large enough, $\sigma(t,x)$ can be approximated well by $\sigma_{R}(t,x)$ defined in \eqref{eq:13a}, and in particular, the estimation \eqref{eq:36} holds in the proof of Theorem \ref{thm:4} in Section \ref{sec:3}.

The main result in this section is stated as follows:
\begin{theorem}
	With Assumptions \textbf{(A1)} to \textbf{(A4)}, there exists a constant $C > 0$ which only depends on $T$, $n$, $d$ and $L$, such that
	\begin{equation}
		\sup\limits_{0\leq t\leq T}\tilde{E}\int_{B_{\sqrt{R}}}|\sigma(t,x)-\sigma_{R}(t,x)|dx\leq \frac{C}{1+R^{n}}
	\end{equation}
	holds for large enough $R$ (for example, $R > 5$), where $\sigma_{R}(t,x)$ is the solution of the IBV problem \eqref{eq:13a}.
	\label{thm:2}
\end{theorem}

\begin{proof}[Proof of Theorem \ref{thm:2}]
	For each $R > 0$, consider the auxiliary function 
	\begin{equation}
		\phi(x) = \log\left(1+R^{n}\left(1-\left(1-\frac{|x|^{2n}}{R^{2n}}\right)^{2}\right)\right),\quad x\in B_{R},
	\end{equation}
	and
	\begin{equation}
		\psi(x) = e^{-\phi(x)} - e^{-\phi(R)},\quad x\in B_{R}.
	\end{equation}
	Define $v(t,x) = \sigma(t,x) - \sigma_{R}(t,x)$, $(t,x)\in[0,T]\times B_{R}$. Then, according to the maximum principle for SPDEs (cf. \cite{CHEKROUN20162926}, for example), we have $v(t,x)\geq 0$, for all $(t,x)\in[0,T]\times B_{R}$ and a.s. $\tilde{P}$. Let $\Phi(t)$ be the stochastic process defined by
	\begin{equation}
		\Phi(t) = \int_{B_{R}}\psi(x)v(t,x)dx.
	\end{equation}
	Since $v(t,x)$ is the solution to the SPDE
	\begin{equation}
		\begin{aligned} 
			dv(t,x) = \biggl[\frac{1}{2}\sum_{i,j=1}^{d}\frac{\partial^{2}}{\partial x_{i}\partial x_{j}}(a^{ij}(x)v(t,x)) &- \sum_{i=1}^{d}\frac{\partial }{\partial x_{i}}(f_{i}(x)v(t,x))\biggr]dt\\
			&+ \sum_{j=1}^{d}h_{j}(x)v(t,x)dY_{t}^{j},
		\end{aligned}
	\end{equation}
	the $\mathbb{R}$-valued stochastic process $\Phi(t)$ satisfies 
	\begin{equation}
		\begin{aligned} 
			d\Phi(t) =& \frac{1}{2}\biggl(\sum_{i,j=1}^{d}\int_{B_{R}}\psi(x)\frac{\partial^{2}}{\partial x_{i}\partial x_{j}}(a^{ij}(x)v(t,x))dx\biggr)dt\\
			&-\biggl(\sum_{i=1}^{d}\int_{B_{R}}\psi(x)\frac{\partial}{\partial x_{i}}(f_{i}(x)v(t,x))dx\biggr)dt\\
			&+\sum_{j=1}^{d}\biggl(\int_{B_{R}}h_{j}(x)\psi(x)v(t,x)dx\biggr)dY_{t}^{j}.
		\end{aligned}
	\end{equation}
	According to the Gauss-Green formula, we have
	\begin{align*}
		d\Phi(t) = & \frac{1}{2}\left(\sum_{i,j=1}^{d}\int_{B_{R}}a^{ij}(x)\frac{\partial^{2}\psi}{\partial x_{i}\partial x_{j}}v(t,x)dx\right)dt\\
		&+\left(\sum_{i=1}^{d}\int_{B_{R}}\frac{\partial\psi}{\partial x_{i}}f_{i}(x)v(t,x)dx\right)dt\\
		&+\sum_{j=1}^{d}\left(\int_{B_{R}}h_{j}(x)\psi(x)v(t,x)dx\right)dY_{t}^{j}\\
		&+\biggl[-\int_{\partial B_{R}}\vec{\mathfrak{M}}_{1}(t,x)\cdot\vec{\mathfrak{n}}dS + \int_{\partial B_{R}}\vec{\mathfrak{M}}_{2}(t,x)\cdot \vec{\mathfrak{n}}dS\biggr]dt,
	\end{align*}
	where, as in the proof of Theorem \ref{thm:1b},
	\begin{equation}
		\vec{\mathfrak{M}}_{i}(t,x) = \left(\mathfrak{M}_{i,1}(t,x),\cdots,\mathfrak{M}_{i,d}(t,x)\right),\quad i = 1,2,
	\end{equation}
	\begin{equation}
		\begin{aligned}
			\mathfrak{M}_{1,j}(t,x) &= \frac{1}{2}\sum_{i=1}^{d}\frac{\partial \psi}{\partial x_{i}}a^{ij}(x)v(t,x),\ j=1,\cdots,d,\\
			\mathfrak{M}_{2,i}(t,x) &= \psi\biggl(\frac{1}{2}\sum_{j=1}^{d}\frac{\partial}{\partial x_{j}}(a^{ij}(x)v(t,x)) - f_{i}(x)v(t,x)\biggr), \ i=1,\cdots,d,
		\end{aligned}
	\end{equation}
	$\vec{\mathfrak{n}}$ denotes the outward normal vector of the boundary $\partial B_{R}$ and $dS$ denotes the measure on $\partial B_{R}$. 
	
	Notice that $\psi|_{\partial B_{R}}\equiv 0$ and 
	\begin{equation}
		\frac{\partial\psi}{\partial x_{j}} = -e^{-\phi(x)}\frac{\partial \phi}{\partial x_{j}}. 
	\end{equation}
	Moreover, 
	\begin{equation}
		\frac{\partial \phi}{\partial x_{i}} = \frac{2R^{n}\biggl(1-\frac{|x|^{2n}}{R^{2n}}\biggr)\frac{2n|x|^{2n-2}x_{i}}{R^{2n}}}{1+R^{n}\biggl(1-\biggl(1-\frac{|x|^{2n}}{R^{2n}}\biggr)^{2}\biggr)},
	\end{equation}
	and therefore,
	\begin{equation}
		\frac{\partial \phi}{\partial x_{i}}\biggr|_{\partial B_{R}} = 0 = \frac{\partial\psi}{\partial x_{i}}\biggr|_{\partial B_{R}},\ i=1,\cdots,d.
	\end{equation}
	Hence,
	\begin{equation}
		\begin{aligned}
			d\Phi(t) = & \frac{1}{2}\left(\int_{B_{R}}e^{-\phi(x)}v(t,x)\sum_{i,j=1}^{d}a^{ij}(x)\left(-\frac{\partial^{2}\phi(x)}{\partial x_{i}\partial x_{j}} + \frac{\partial \phi}{\partial x_{i}}\frac{\partial\phi}{\partial x_{j}}\right)dx\right)dt\\
			&- \biggl(\int_{B_{R}}\sum_{i=1}^{d}e^{-\phi(x)}v(t,x)f_{i}(x)\frac{\partial\phi}{\partial x_{i}}dx\biggr)dt\\
			&+ \sum_{j=1}^{d}\left(\int_{B_{R}}h_{j}(x)\psi(x)v(t,x)dx\right)dY_{t}^{j}
		\end{aligned}
	\end{equation}
	Take expectation with respect to the probability measure $\tilde{P}$, and we have
	\begin{equation}
		\begin{aligned}
			\frac{d\tilde{E}\Phi(t)}{dt} = & \frac{1}{2}\tilde{E}\left(\int_{B_{R}}\psi(x)v(t,x)\sum_{i,j=1}^{d}a^{ij}(x)\left(-\frac{\partial^{2}\phi(x)}{\partial x_{i}\partial x_{j}} + \frac{\partial \phi}{\partial x_{i}}\frac{\partial\phi}{\partial x_{j}}\right)dx\right)\\
			&- \tilde{E}\biggl(\int_{B_{R}}\sum_{i=1}^{d}\psi(x)v(t,x)f_{i}(x)\frac{\partial\phi}{\partial x_{i}}dx\biggr)\\
			& +e^{-\phi(R)}\frac{1}{2}\tilde{E}\left(\int_{B_{R}}v(t,x)\sum_{i,j=1}^{d}a^{ij}(x)\left(-\frac{\partial^{2}\phi(x)}{\partial x_{i}\partial x_{j}} + \frac{\partial \phi}{\partial x_{i}}\frac{\partial\phi}{\partial x_{j}}\right)dx\right)\\
			&- e^{-\phi(R)}\tilde{E}\biggl(\int_{B_{R}}\sum_{i=1}^{d}v(t,x)f_{i}(x)\frac{\partial\phi}{\partial x_{i}}dx\biggr)
		\end{aligned}
	\end{equation}
	For $x\in B_{R}$, $|x_{i}|\leq |x|\leq R$, $i = 1,\cdots,d$, and together with the Lipschitz conditions for $f(x)$, 
	\begin{equation}
		\left|\frac{\partial \phi}{\partial x_{i}}(x)\right|\leq \frac{4nR^{n}|x|^{2n-2}|x_{i}|}{R^{2n}\biggl(1+\frac{|x|^{2n}}{R^{n}}\biggr)}\leq 4n,\ \forall \ x\in B_{R};
	\end{equation}
	\begin{equation}
		\begin{aligned} 
			\biggl|f_{i}(x)\frac{\partial \phi}{\partial x_{i}}\biggr|&\leq (|f(0)|+L|x|)\biggl|\frac{\partial \phi}{\partial x_{i}}\biggr|\\
			&\leq 4n|f(0)| + 4nL\frac{R^{n}|x|^{2n-1}|x_{i}|}{R^{2n}\biggl(1+\frac{|x|^{2n}}{R^{n}}\biggr)}\\
			&\leq 4n(|f(0)|+L),\ \forall\ x\in B_{R}.
		\end{aligned}
	\end{equation}
	
	Also, according to direct computations,
	\begin{equation}
		\begin{aligned} 
			\frac{\partial^{2}\phi}{\partial x_{i}\partial x_{j}} = & \frac{8n|x|^{2n-4}x_{i}x_{j}(R^{2n}(n-1)-(2n-1)|x|^{2n})(R^{3n}+2R^{2n}|x|^{2n}-|x|^{4n})}{(R^{3n}+2R^{2n}|x|^{2n}-|x|^{4n})^{2}}\\
			&-\frac{16n^{2}|x|^{2n-2}x_{i}x_{j}(R^{2n}|x|^{2n-2}-|x|^{4n-2})(R^{2n}-|x|^{2n})}{(R^{3n}+2R^{2n}|x|^{2n}-|x|^{4n})^{2}}\\
			&+\frac{4n\delta_{ij}|x|^{2n-2}(R^{2n}-|x|^{2n})}{R^{3n}+2R^{2n}|x|^{2n}-|x|^{4n}}.
		\end{aligned}
		\label{eq:80}
	\end{equation}
	where $\delta_{ij}$ is the Kronecker's symbol. Thus,
	\begin{equation}
		\biggl|\frac{\partial^{2}\phi}{\partial x_{i}\partial x_{j}}\biggr|\leq 8n(3n-2) + 16n^{2} + 4n,\ \forall x\in B_{R}.
		\label{eq:81}
	\end{equation}
	We would like to remark that the estimation in \eqref{eq:81} is quite rough. Each term on the right-hand side of \eqref{eq:81} corresponds to one term on the right-hand side of \eqref{eq:80}, and the purpose is just to show the second-order derivatives are also bounded by a constant independent of $R$. 
	
	Notice that $e^{-\phi(R)} = \frac{1}{1+R^{n}}$.	Together with the bounded condition for $a^{ij}(x)$, we have
	\begin{equation}
		\frac{d\tilde{E}\Phi(t)}{dt} \leq C_{1}\tilde{E}\Phi(t) + \frac{C_{1}}{1+ R^{n}}\tilde{E}\int_{B_{R}}v(t,x)dx
		\label{eq:43}
	\end{equation}
	where $C_{1}> 0$ is a constant which depends on $n,d,L$, but does not depend on $R$.
	
	According to Theorem \ref{thm:1b}, the integral
	\begin{equation}
		\tilde{E}\int_{B_{R}}v(t,x)dx\leq \tilde{E}\int_{B_{R}}\sigma(t,x)dx\leq \tilde{E}\int_{\mathbb{R}^{d}}\sigma(t,x)dx,
	\end{equation}
	which is also bounded by a constant independent of $R$, thus,
	\begin{equation}
		\frac{d\tilde{E}\Phi(t)}{dt}\leq C_{1}\tilde{E}\Phi(t) + \frac{C_{2}}{1+R^{n}}. 
	\end{equation}
	where $C_{2}> 0$ is a constant which depends on $T,n,d,L$.
	
	By Gronwall's inequality,
	\begin{equation}
		\tilde{E}\Phi(t)\leq \frac{C_{3}}{1+R^{n}}
		\label{eq:45}
	\end{equation}
	where $C_{3}> 0$ is a constant which depends on $T$, $n$, $d$ and $L$. 
	
	On the other hand, for $R > 5$ and for all $x\in B_{\sqrt{R}}$, i.e., $|x|\leq \sqrt{R}$, 
	\begin{equation}
		\phi(x)\in\biggl[0,\log\left(3-\frac{1}{R^{n}}\right)\biggr],\quad \psi(x)\geq\psi(\sqrt{R}) = \frac{R^{n}}{3R^{n}-1}-\frac{1}{1+R^{n}}\geq \frac{1}{3}-\frac{1}{6} = \frac{1}{6}.
	\end{equation} 
	Then,
	\begin{equation}
		\tilde{E}\Phi(t) = \tilde{E}\int_{B_{R}}\psi(x)v(t,x)dx\geq \tilde{E}\int_{B_{\sqrt{R}}}\psi(x)v(t,x)dx\geq \frac{1}{6}\tilde{E}\int_{B_{\sqrt{R}}}v(t,x)dx.
		\label{eq:47}
	\end{equation}
	Combining \eqref{eq:45} and \eqref{eq:47}, we obtain that, for all $R\gg 1$, 
	\begin{equation}
		\tilde{E}\int_{B_{\sqrt{R}}}|\sigma(t,x)-\sigma_{R}(t,x)|dx = \tilde{E}\int_{B_{\sqrt{R}}}v(t,x)dx\leq \frac{6C_{3}}{1+R^{n}}. 
	\end{equation}
\end{proof}

\section{Regularity of the Approximated Function \texorpdfstring{$u_{k}(t,x)$}{uktx}}
\label{sec:6}
In this section, we will discuss the regularity of $u_{k}(t,x)$, $t\in [0,T]$, which is the solution of a series of coefficient-frozen equations \eqref{eq:5}.

The main purpose of this section is to show that under mild conditions, the recursively defined functions $u_{k}(t,x)$ will not explode in the finite time interval $[0,T]$, even if the time-discretization step $\delta\rightarrow 0$, in the sense that the $L^{2}$-norm of $u_{k}(\tau_{k},x)$ ($k=1,\cdots,K$) is square integrable with respect to the probability measure $\tilde{P}$, and the expectations, $\tilde{E}\int_{B_{R}}|u_{k}(\tau_{k},x)|^{2}dx$, are uniformly bounded for $k=1,\cdots, K$. 

As shown in the next section, this following theorem is an essential intermediate result for the convergence analysis of this time-discretization scheme. 
\begin{theorem} 
	Let $\{u_{k}(t,x):\tau_{k-1}\leq t\leq \tau_{k}\}_{k=1}^{K}$ be the solution to the IBV problem of the coefficients-frozen equation \eqref{eq:5}. Then, with Assumptions \textbf{(A1)} to \textbf{(A4)},  the $L^{2}$-norm of $u_{k}(\tau_{k},x)$ is square-integrable with respect to the probability measure $\tilde{P}$, and we have
	\begin{equation}
		\tilde{E}\int_{B_{R}}|u_{k}(\tau_{k},x)|^{2}dx \leq C < \infty,\ \forall\ 
		k=1,\cdots,K,
	\end{equation}
	where $C > 0$ is a constant that depends on $d$, $T$, $R$, $L$, but is uniform in $k=1,\cdots,K$.
	\label{thm:4a}
\end{theorem} 
In the proof of Theorem \ref{thm:4a}, we will consider another exponential transformation given by
\begin{equation}
	\sigma_{k}(t,x) = \exp(h^{\top}(x)Y_{\tau_{k-1}})u_{k}(t,x),\ t\in[\tau_{k-1},\tau_{k}],\ k=1,\cdots,K.
\end{equation}
Direct computation implies that $\sigma_{k}(t,x)$ is the solution of 
\begin{equation}
	\left\{\begin{aligned}
		&\frac{\partial\sigma_{k}(t,x)}{\partial t} = \frac{1}{2}\sum_{i,j=1}^{d}\frac{\partial^{2}}{\partial x_{i}\partial x_{j}}(a^{ij}(x)\sigma_{k}(t,x)) - \sum_{i=1}^{d}\frac{\partial}{\partial x_{i}}(f_{i}(x)\sigma_{k}(t,x)) \\
		&\qquad- \frac{1}{2}|h(x)|^{2}\sigma_{k}(t,x),\ (t,x)\in[\tau_{k-1},\tau_{k}]\times B_{R},\\
		&\sigma_{k}(\tau_{k-1},x) = \exp\left(h^{\top}(x)Y_{\tau_{k-1}}\right)u_{k-1}(\tau_{k-1},x),\ x\in B_{R}\\
		&\sigma_{k}(t,x) = 0,\ (t,x)\in[\tau_{k-1},\tau_{k}]\times\partial B_{R},
	\end{aligned}\right.
	\label{eq:98}
\end{equation}
and recursively, we can rewrite the initial value in \eqref{eq:98} by
\begin{equation}
	\sigma_{k}(\tau_{k-1},x) = \exp\left(h^{\top}(x)(Y_{\tau_{k-1}}-Y_{\tau_{k-2}})\right)\sigma_{k-1}(\tau_{k-1},x),\ k=2,\cdots,K.
	\label{eq:99}
\end{equation}

Under the reference probability measure $\tilde{P}$, $\{Y_{t}:0\leq t\leq T\}$ is a Brownian motion and 
\begin{equation}
	Y_{\tau_{k}} - Y_{\tau_{k-1}}\sim\mathcal{N}(0,\delta I_{d}),\ k=1,\cdots,K,
\end{equation} 
with $I_{d}\in\mathbb{R}^{d\times d}$ the $d$-dimensional identity matrix. We would like to study the regularity of $\sigma_{k}(t,x)$ first, utilizing the Markov property of $Y$, and then derive the regularity results for $u_{k}(t,x)$.

For the sake of discussing the regularity of $\sigma_{k}(t,x)$ in a recursive manner, we need the following lemma which describes the relationship between $\sigma_{k}(\tau_{k-1},x)$ and $\sigma_{k-1}(\tau_{k-1},x)$ from \eqref{eq:99}.
\begin{lemma}
	For $k=2,\cdots,K$, let $\sigma_{k}(t,x)$, $t\in[\tau_{k-1},\tau_{k}]$ be the solution of \eqref{eq:98}. The end-point values $\sigma_{k}(\tau_{k-1},x)$ and $\sigma_{k-1}(\tau_{k-1},x)$ satisfy \eqref{eq:99}. Let us denote by $L^{4}(B_{R})$ the space of quartic-integrable functions in $B_{R}$. Assume that $\sigma_{k-1}(\tau_{k-1},\cdot)\in L^{4}(B_{R})$, and the $L^{4}$-norm, $\|\sigma_{k-1}(\tau_{k-1},\cdot)\|_{L^{4}}$, is quartic integrable with respect to $\tilde{P}$, i.e.,
	\begin{equation}
		\tilde{E}\int_{B_{R}}\sigma_{k-1}^{4}(\tau_{k-1},x)dx < \infty
	\end{equation}
	then $\sigma_{k}(\tau_{k-1},\cdot)\in L^{4}(B_{R})$, its $L^{4}$-norm, $\|\sigma_{k}(\tau_{k-1},\cdot)\|_{L^{4}}$ is quartic integrable with respect to $\tilde{P}$, and for sufficiently small time-discretization step size $\delta = \tau_{k}-\tau_{k-1}$, we have
	\begin{equation}
		\tilde{E}\int_{B_{R}}\sigma_{k}^{4}(\tau_{k-1},x)dx\leq (1+C\delta)\tilde{E}\int_{B_{R}}\sigma_{k-1}^{4}(\tau_{k-1},x)dx
	\end{equation}
	where $C$ is a constant that depends on $d$ and $R$. 
	\label{lem:1}
\end{lemma}
\begin{proof}[Proof of Lemma \ref{lem:1}]
	According to the expression \eqref{eq:99} and the definition of $\sigma_{k-1}$ on $[\tau_{k-2},\tau_{k-1}]$, because of the Markov property of $Y$, $\exp\left(h^{\top}(x)(Y_{\tau_{k-1}}-Y_{\tau_{k-2}})\right)$ is independent of $\sigma_{k-1}(\tau_{k-1},x)$.
	
	Because the observation function $h$ is assumed to be smooth enough, and $B_{R}$ is a bounded domain in $\mathbb{R}^{d}$, there exists a constant $M$, which may depend on $R$, such that the maximum of the absolute value of $h$, together with its partial derivatives up to order $m$, is bounded above by $M$.
	
	Therefore, by Fubini's theorem,
	\begin{equation}
		\begin{aligned} 
			\tilde{E}\int_{B_{R}}\sigma_{k}^{4}(\tau_{k-1},x)dx &= \tilde{E}\int_{B_{R}}\exp\left(4h^{\top}(x)(Y_{\tau_{k-1}}-Y_{\tau_{k-2}})\right)\sigma_{k-1}^{4}(\tau_{k-1},x)dx\\
			&=\int_{B_{R}} \tilde{E}\exp\left(4h^{\top}(x)(Y_{\tau_{k-1}}-Y_{\tau_{k-2}})\right)\tilde{E}\sigma_{k-1}^{4}(\tau_{k-1},x)dx.
		\end{aligned} 
	\end{equation}

			Next, let us estimate the expectations of functions of normal random variable $\xi:=Y_{\tau_{k-1}} - Y_{\tau_{k-2}}$ arising in the above expressions, for small time-discretization step $\delta$.
			
			In fact, because $\xi\sim\mathcal{N}(0,\delta I_{d})$, we have
			\begin{equation}
				\tilde{E}\exp\left(4h(x)^{\top}\xi\right) = \prod_{j=1}^{d}\tilde{E}e^{4h_{j}(x)\xi_{j}} =\prod_{j=1}^{d} \left(\int_{\mathbb{R}}\frac{1}{\sqrt{2\pi\delta}}e^{4h_{j}(x)z}e^{-\frac{z^{2}}{2\delta}}dz\right)
			\end{equation}
			In the bounded domain $B_{R}$,
			\begin{equation}
				\begin{aligned} 
					\int_{\mathbb{R}}\frac{1}{\sqrt{2\pi\delta}}e^{4h_{j}(x)z}e^{-\frac{z^{2}}{2\delta}}dz &= \int_{\mathbb{R}}\frac{1}{\sqrt{2\pi}}e^{4h_{j}(x)\sqrt{\delta}z}e^{-\frac{z^{2}}{2}}dz\\
					&= e^{8h_{j}^{2}(x)\delta}\int_{\mathbb{R}}\frac{1}{\sqrt{2\pi}}e^{-\frac{1}{2}(z-4h_{j}(x)\sqrt{\delta})^{2}}dz\\
					&=e^{8h_{j}^{2}(x)\delta}\leq e^{8M^{2}\delta}.
				\end{aligned}
			\end{equation}
			Therefore, for $\delta\ll 1$ (for example $\delta\leq \frac{1}{16M^{2}d}$),
			\begin{equation}
				\tilde{E}\exp\left(4h(x)^{\top}\xi\right)\leq e^{8dM^{2}\delta}\leq 1 + 16dM^{2}\delta.
			\end{equation}
			Thus,
			\begin{equation}
				\tilde{E}\int_{B_{R}}\sigma_{k}^{4}(\tau_{k-1},x)dx\leq (1+16dM^{2}\delta)\tilde{E}\int_{B_{R}}\sigma_{k-1}^{4}(\tau_{k-1},x)dx.
			\end{equation}
			\end{proof}
			
			Now, we are ready to give the proof of Theorem \ref{thm:4a}.
			\begin{proof}[Proof of Theorem \ref{thm:4a}]
				The idea of this proof is to study the regularity of $\sigma_{k}(t,x)$, recursively, and then obtain the regularity of $u_{k}(t,x)$ based on the relationship \eqref{eq:99}.
				
				In fact, according to the Cauchy-Schwartz inequality,
				\begin{align*}
					\tilde{E}\int_{B_{R}}|u_{k}(\tau_{k},x)|^{2}dx &= \tilde{E}\int_{B_{R}}\exp\left(-2h^{\top}(x)Y_{\tau_{k-1}}\right)\sigma_{k}^{2}(\tau_{k},x)dx\\
					&\leq \tilde{E}\biggl[\exp\left(2M\sum_{j=1}^{d}|Y_{\tau_{k-1},j}|\right)\int_{B_{R}}\sigma_{k}^{2}(\tau_{k},x)dx\biggr]\\
					&\leq \biggl(\tilde{E}\exp\left(4M\sum_{j=1}^{d}|Y_{\tau_{k-1},j}|\right)\biggr)^{\frac{1}{2}}\biggl(\tilde{E}\biggl(\int_{B_{R}}\sigma_{k}^{2}(\tau_{k},x)dx\biggr)^{2}\biggr)^{\frac{1}{2}}\\
					&\leq C_{1}\biggl(\tilde{E}\exp\left(4M\sum_{j=1}^{d}|Y_{\tau_{k-1},j}|\right)\biggr)^{\frac{1}{2}}\biggl(\tilde{E}\int_{B_{R}}\sigma_{k}^{4}(\tau_{k},x)dx\biggr)^{\frac{1}{2}}
				\end{align*}
				with $C_{1}> 0$, a constant depending only on $R$.
				
				Under the reference probability measure $\tilde{P}$, $\{Y_{t}:0\leq t\leq T\}$ is a standard $d$-dimensional Brownian motion, and therefore, the expectation $$\tilde{E}\exp\left(4M\sum_{j=1}^{d}|Y_{\tau_{k-1},j}|\right)$$ is bounded.
				
				Hence, it remains to show that there exists a constant $C_{2}>0$, such that,
				\begin{equation}
					\tilde{E}\int_{B_{R}}\sigma_{k}^{4}(\tau_{k},x)dx \leq C_{2} < \infty,
				\end{equation}
				holds uniformly for $k=1,\cdots,K$.
				
				In the time interval $[\tau_{k-1},\tau_{k}]$, $\sigma_{k}(t,x)$ is the solution to \eqref{eq:98}. According to the regularity results of parabolic partial differential equations, we have
				\begin{equation}
					\int_{B_{R}}\sigma_{k}^{4}(\tau_{k},x)dx\leq e^{C_{4}\delta}\int_{B_{R}}\sigma_{k}^{4}(\tau_{k-1},x)dx,\ \forall\ k=1,\cdots,K.
					\label{eq:110}
				\end{equation}
				where $C_{4}$ is a constant which depends on the coefficients of the filtering system. The techniques in the proof of \eqref{eq:110} is standard, and the proof of a counterpart, in which $L^{2}$-norm (instead of $L^{4}$-norm) is considered, can be found in the textbook \cite{RefWorks:RefID:3-evans2010partial}. We also provide a detailed proof in the Appendix, for the readers' convenience and in order to keep this paper self-contained.
				
				Thus, with the result in Lemma \ref{lem:1}, there exists $C_{5},C_{6}>0$, such that for small enough $\delta$,
				\begin{equation}
					\begin{aligned} 
						\tilde{E}\int_{B_{R}}\sigma_{k}^{4}(\tau_{k},x)dx&\leq e^{C_{4}\delta}\tilde{E}\int_{B_{R}}\sigma_{k}^{4}(\tau_{k-1},x)dx\leq e^{C_{4}\delta}(1+C_{5}\delta)\tilde{E}\int_{B_{R}}\sigma_{k-1}^{4}(\tau_{k-1},x)dx\\
						&\leq (1+C_{6}\delta)\tilde{E}\int_{B_{R}}\sigma_{k-1}^{4}(\tau_{k-1},x)dx.
					\end{aligned} 
				\end{equation}
				Inductively, we have
				\begin{equation}
					\begin{aligned} 
						\tilde{E}\int_{B_{R}}\sigma_{k}^{4}(\tau_{k},x)dx&\leq (1+C_{6}\delta)^{k}\int_{B_{R}}\sigma_{0}^{4}(x)dx\\
						&\leq (1+C_{6}\delta)^{\frac{T}{\delta}}\int_{B_{R}}\sigma_{0}^{4}(x)dx\leq e^{C_{6}T}\int_{B_{R}}\sigma_{0}^{4}(x)dx.
					\end{aligned} 
				\end{equation}
				Thus, we have proved the boundedness of $\tilde{E}\int_{B_{R}}\sigma_{k}^{4}(\tau_{k},x)dx$, and also, the result of Theorem \ref{thm:4a} holds.
			\end{proof}
			
			\section{Convergence Analysis of the Time Discretization Scheme}
			\label{sec:7}
			This section serves to show that the solution $u_{k}(t,x)$ of the coefficient-frozen equations \eqref{eq:5} can approximate the solution $u(t,x)$ of the original robust DMZ equation \eqref{eq:4} well, if the time-discretization step size $\delta$ is small enough. 
			
			Also, we will show in this section that, after the exponential transformation $\exp(h^{\top}(x)Y_{\tau_{k}})$, the $L^{1}$-norm of the difference between the unnormalized densities $\sigma_{R}(\tau_{k},x)$ (defined by \eqref{eq:13a}) and $\tilde{u}_{k+1}(\tau_{k},x)$ (defined by \eqref{eq:6}) still converges to zero, as $\delta\rightarrow 0$. In particular, the estimation \eqref{eq:38} holds in the proof of Theorem \ref{thm:4} in Section \ref{sec:3}.
			\begin{theorem}
				Fix $R > 0$. With Assumptions \textbf{(A1)} to \textbf{(A4)}, we can use the solution $u_{k}(t,x)$ of equation \eqref{eq:5} to approximate the solution $u(t,x)$ of equation \eqref{eq:4}. In particular, for every $\epsilon > 0$, there exists a constant $\delta > 0$, such that
				\begin{equation}
					\tilde{E}\int_{B_{R}}|\sigma_{R}(\tau_{k},x)-\tilde{u}_{k+1}(\tau_{k},x)|dx = \tilde{E}\int_{B_{R}}e^{h^{\top}(x)Y_{\tau_{k}}}\left|u(\tau_{k},x) - u_{k}(\tau_{k},x)\right|dx < \epsilon,
					\label{eq:64}
				\end{equation}
				holds for every $k=1,\cdots,K$.
				\label{thm:1a}
			\end{theorem}
			
			\begin{proof}[Proof of Theorem \ref{thm:1a}]
				Since $f$ is globally Lipschitz, $h\in C^{2}(B_{R})$, and $B_{R}$ is a bounded domain, there exists a constant $M_{0}>0$, such that the absolute value of each component in $f(x)$ and $h(x)$, as well as there first and second order derivatives, are dominated by $M$ in the ball $B_{R}$, i.e.,
				\begin{equation}
					\begin{aligned} 
						\max\limits_{x\in B_{R}}\biggl\{\max\limits_{1\leq i\leq d}|f_{i}(x)|,&\max\limits_{1\leq i\leq d}|h_{i}(x)|,\max\limits_{1\leq i,j\leq d}\biggl|\frac{\partial f_{i}}{\partial x_{j}}\biggr|,\\
						&\max\limits_{1\leq i,j\leq d}\biggl|\frac{\partial h_{i}}{\partial x_{j}}\biggr|,\max\limits_{1\leq i,j,k\leq d}\biggl|\frac{\partial^{2} f_{i}}{\partial x_{j}\partial x_{k}}\biggr|,\max\limits_{1\leq i\leq d}\biggl|\frac{\partial^{2} h_{i}}{\partial x_{j}\partial x_{k}}\biggr|\biggr\}\leq M_{0}.
					\end{aligned}
				\end{equation}
				
				Let $B_{R,t}^{+} = \{x\in B_{R}:u(t,x)-u_{k}(t,x)\geq 0\}$. According to the technical Lemma 4.1 in \cite{yymethod}, we have
				\begin{equation}
					\frac{d}{dt}\int_{B_{R,t}^{+}}(u(t,x)-u_{k}(t,x))dx = \int_{B_{R,t}^{+}}\frac{\partial}{\partial t}(u(t,x)-u_{k}(t,x))dx,
				\end{equation}
				for almost all $t\in [0,T]$.
				
				Then, according to equations \eqref{eq:4} and \eqref{eq:5} satisfied by $u(t,x)$ and $u_{k}(t,x)$ in $[\tau_{k-1},\tau_{k}]$,
				\begin{equation}
					\begin{aligned}
						&\frac{d}{dt}\int_{B_{R,t}^{+}}(u(t,x)-u_{k}(t,x))dx \\
						&= \int_{B_{R,t}^{+}}\frac{\partial}{\partial t}(u(t,x)-u_{k}(t,x))dx\\
						&=\frac{1}{2}\int_{B_{R,t}^{+}}\sum_{i,j=1}^{d}a^{ij}(x)\frac{\partial^{2}}{\partial x_{i}\partial x_{j}}(u-u_{k})dx + \int_{B_{R,t}^{+}}\sum_{i=1}^{d}F_{i}(\tau_{k-1},x)\frac{\partial }{\partial x_{i}}(u-u_{k})dx\\
						&\quad +\int_{B_{R,t}^{+}}J(\tau_{k-1},x)(u(t,x)-u_{k}(t,x))dx\\
						&\quad + \int_{B_{R,t}^{+}}\sum_{i=1}^{d}(F_{i}(t,x)-F_{i}(\tau_{k-1},x))\frac{\partial u}{\partial x_{i}}dx + \int_{B_{R,t}^{+}}(J(t,x)-J(\tau_{k-1},x))u(t,x)dx.
					\end{aligned}	
					\label{eq:113}
				\end{equation}
				Because $u(t,x) = u_{k}(t,x)\equiv 0$ on the boundary $\partial B_{R}$, and $\partial B_{R,t}^{+}\subset\partial B_{R}\cup\{x\in B_{R}:u(t,x)-u_{k}(t,x) = 0\}$, we have $(u-u_{k})|_{\partial B_{R,t}^{+}} = 0$ and $\nabla(u-u_{k})|_{\partial B_{R,t}^{+}} = -c(x)\vec{\mathfrak{n}}$ with $\vec{\mathfrak{n}}$ the outward normal vector of $B_{R,t}^{+}$ and $c(x)\geq 0$ on $\partial B_{R,t}^{+}$. Thus, the first three terms on the right-hand side of \eqref{eq:113} can be estimated by
				\begin{equation}
					\begin{aligned}
						\frac{1}{2}&\int_{B_{R,t}^{+}}\sum_{i,j=1}^{d}a^{ij}\frac{\partial^{2}}{\partial x_{i}\partial x_{j}}(u-u_{k})dx +\int_{B_{R,t}^{+}}\sum_{i=1}^{d}F_{i}(\tau_{k-1},x)\frac{\partial }{\partial x_{i}}(u-u_{k})dx\\
						&+\int_{B_{R,t}^{+}}J(\tau_{k-1},x)(u(t,x)-u_{k}(t,x))dx\\
						&= \frac{1}{2}\int_{B_{R,t}^{+}}\sum_{i,j=1}^{d}\frac{\partial^{2}a^{ij}}{\partial x_{i}\partial x_{j}}(u-u_{k})dx-\int_{B_{R,t}^{+}}\sum_{i=1}^{d}\frac{\partial F_{i}(\tau_{k-1},x)}{\partial x_{i}}(u-u_{k})dx\\
						&\quad+\int_{B_{R,t}^{+}}J(\tau_{k-1},x)(u(t,x)-u_{k}(t,x))dx\\
						&-\frac{1}{2}\int_{\partial B_{R,t}^{+}}\sum_{i,j=1}^{d}a^{ij}\frac{\partial}{\partial x_{i}}(u-u_{k})\frac{\partial}{\partial x_{j}}(u-u_{k})dS\\
						&\quad -\frac{1}{2}\int_{\partial B_{R,t}^{+}}(u-u_{k})\sum_{i,j=1}^{d}\frac{\partial a^{ij}}{\partial x_{i}}\vec{\mathfrak{n}}_{j}dS + \int_{\partial B_{R,t}^{+}}(u-u_{k})\sum_{i=1}^{d}F_{i}(\tau_{k-1},x)\vec{\mathfrak{n}}_{i}dS\\
						&\leq \biggl(\frac{d^{2}L}{2}+C(d,L,M_{0})\biggl(1+\sum_{j=1}^{d}|Y_{\tau_{k-1},j}|\biggr)^{2}\biggr)\int_{B_{R,t}^{+}}(u-u_{k})dx,\\
						&\leq C_{1}\biggl(1+\sum_{j=1}^{d}|Y_{\tau_{k-1},j}|\biggr)^{2}\int_{B_{R,t}^{+}}(u-u_{k})dx
					\end{aligned}
				\end{equation}
				where we use the fact that $a(x) = g(x)g(x)^{\top}$ is positive semi-definite and the definition of $F_{i}(t,x)$ and $J(t,x)$ in \eqref{eq:16}; $C(d,L,M_{0})$ and $C_{1}$ are constants which depend only on $d,L,M_{0}$; and $dS$ denotes the measure on $\partial B_{R,t}^{+}$.
				
				Also, by the definition of $F_{i}(t,x)$ and $J(t,x)$ in \eqref{eq:16}, we have the following estimation of the differences
				\begin{equation}
					\begin{aligned}
						&|F(t,x) - F(\tau_{k-1},x)|\leq C_{2}|Y_{t}-Y_{\tau_{k-1}}|,\\
						&|J(t,x) - J(\tau_{k-1},x)|\leq C_{3}\biggl(1+\sum_{j=1}^{d}(|Y_{t,j}|+|Y_{\tau_{k-1},j}|)\biggr)|Y_{t}-Y_{\tau_{k}}|,\ \forall\ x\in B_{R},
					\end{aligned}
				\end{equation}
				where $C_{2}$ and $C_{3}$ are constants which only depends on $d,L.M_{0}$.
				
				Hence,
				\begin{equation}
					\begin{aligned}
						\frac{d}{dt}\int_{B_{R,t}^{+}}&(u(t,x)-u_{k}(t,x))dx \leq C_{1}\biggl(1+\sum_{j=1}^{d}|Y_{\tau_{k-1},j}|\biggr)^{2}\int_{B_{R,t}^{+}}(u-u_{k})dx\\
						&+ C_{2}|Y_{t}-Y_{\tau_{k-1}}|\int_{B_{R}}|\nabla u(t,x)|dx \\
						&+ C_{3}\biggl(1+\sum_{j=1}^{d}(|Y_{t,j}|+|Y_{\tau_{k-1},j}|)\biggr)|Y_{t}-Y_{\tau_{k}}|\int_{B_{R}}|u(t,x)|dx
					\end{aligned}
					\label{eq:66}
				\end{equation}
				holds for almost all $t\in[\tau_{k-1},\tau_{k}]$ and almost surely, where $C_{1}$, $C_{2}$ and $C_{3}$ are constants which depend on the coefficients of the system.  
				
				Under the reference probability distribution $\tilde{P}$, the observation process $\{Y_{t}:0\leq t\leq T\}$ is a standard $d$-dimensional Brownian motion, and therefore, 
				\begin{equation}
					\begin{aligned} 
						\sum_{j=1}^{d}(Y_{t,j}-Y_{\tau_{k-1},j})\sim N(0,d(t-\tau_{k-1})).
					\end{aligned}
				\end{equation}
				
				Let $\Omega_{M_{1}} = \left\{\omega:\sup\limits_{0\leq t\leq T}\sum_{j=1}^{d}|Y_{t,j}(\omega)|\leq M_{1}\right\}$ be the event which represents the observation process $Y_{t}$ is not severely abnormal, and $1_{A}(\cdot)$ is the indicator function of the set $A$.
				
				For a fixed $M_{1}>0$, let us first take the expectation with respect to $\tilde{P}$ on the event $\Omega_{1,M_{1}}$ for both sides of \eqref{eq:66}, and we have
				\begin{align*}
					\frac{d}{dt}\tilde{E}&\biggl[1_{\Omega_{M_{1}}}\int_{B_{R,t}^{+}}(u(t,x)-u_{k}(t,x))dx\biggr] \\
					\leq&\  C_{1}\tilde{E}\biggl[1_{\Omega_{M_{1}}}\biggl(1+\sum_{j=1}^{d}|Y_{\tau_{k-1},j}|\biggr)^{2}\int_{B_{R,t}^{+}}(u-u_{k})dx\biggr]\\
					&+ C_{2}\tilde{E}\biggl[1_{\Omega_{M_{1}}}|Y_{t}-Y_{\tau_{k-1}}|\int_{B_{R}}|\nabla u|dx\biggr] \\
					&+ C_{3}\tilde{E}\biggl[1_{\Omega_{M_{1}}}\biggl(1+\sum_{j=1}^{d}(|Y_{t,j}|+|Y_{\tau_{k-1},j}|)\biggr)|Y_{t}-Y_{\tau_{k}}|\int_{B_{R}}|u|dx\biggr]\\
					\leq&\  C_{1}(1+M_{1})^{2}\tilde{E}\biggl[1_{\Omega_{M_{1}}}\int_{B_{R,t}^{+}}(u(t,x)-u_{k}(t,x))dx\biggr] \\
					&+ C_{2}\tilde{E}\biggl[1_{\Omega_{M_{1}}}|Y_{t}-Y_{\tau_{k-1}}|\int_{B_{R}}|\nabla u|dx\biggr] \\
					&+ C_{3}(1+2M_{1})\tilde{E}\biggl[1_{\Omega_{M_{1}}}|Y_{t}-Y_{\tau_{k-1}}|\int_{B_{R}}|u|dx\biggr]\\
					\leq&\  C_{1}(1+M_{1})^{2}\tilde{E}\biggl[1_{\Omega_{M_{1}}}\int_{B_{R,t}^{+}}(u(t,x)-u_{k}(t,x))dx\biggr] \\
					&+C_{2}\biggl(\tilde{E}|Y_{t}-Y_{\tau_{k-1}}|^{2}\biggr)^{\frac{1}{2}}\left(\tilde{E}\biggl[1_{\Omega_{M_{1}}}\left(\int_{B_{R}}|\nabla u|dx\right)^{2}\biggr]\right)^{\frac{1}{2}}\\
					&+ C_{3}(1+2M_{1})\left(\tilde{E}|Y_{t}-Y_{\tau_{k-1}}|^{2}\right)^{\frac{1}{2}}\left(\tilde{E}\biggl[1_{\Omega_{M_{1}}}\left(\int_{B_{R}}|u|dx\right)^{2}\biggr]\right)^{\frac{1}{2}}\\
					=&\ C_{1}(1+M_{1})^{2}\tilde{E}\biggl[1_{\Omega_{M_{1}}}\int_{B_{R,t}^{+}}(u(t,x)-u_{k}(t,x))dx\biggr] \\
					&+C_{2}d^{\frac{1}{2}}(t-\tau_{k-1})^{\frac{1}{2}}\left(\tilde{E}\biggl[1_{\Omega_{M_{1}}}\left(\int_{B_{R}}|\nabla u|dx\right)^{2}\biggr]\right)^{\frac{1}{2}}\\
					&+ C_{3}(1+2M_{1})d^{\frac{1}{2}}(t-\tau_{k-1})^{\frac{1}{2}}\left(\tilde{E}\biggl[1_{\Omega_{M_{1}}}\left(\int_{B_{R}}|u|dx\right)^{2}\biggr]\right)^{\frac{1}{2}}.
				\end{align*}
				Here, the second inequality holds because of the property of the event $\Omega_{M_{1}}$, the third inequality holds according to the Cauchy-Schwartz inequality and the last equality holds because $Y_{t}$ is a normal distributed random vector.
				
				On the event $\Omega_{M_{1}}$, the observation process $\{Y_{t}:0\leq t\leq T\}$ is bounded. Therefore, according to the regularity results of parabolic partical differential equations (cf. \cite{RefWorks:RefID:3-evans2010partial}, Section 7.1, Theorem 6), the integrals $\int_{B_{R}}|\nabla u|dx$ and $\int_{B_{R}}|u|dx$ are also bounded for almost every $t\in[0,T]$, as long as $f\in C^{1}(B_{R})$ and $h\in C^{2}(B_{R})$. Thus,
				\begin{equation}
					\begin{aligned}
						\frac{d}{dt}\tilde{E}\biggl[1_{\Omega_{M_{1}}}&\int_{B_{R,t}^{+}}(u(t,x)-u_{k}(t,x))dx\biggr] \\&\leq C_{4}\tilde{E}\biggl[1_{\Omega_{M_{1}}}\int_{B_{R,t}^{+}}(u(t,x)-u_{k}(t,x))dx\biggr]
						+ C_{5}(t-\tau_{k-1})^{\frac{1}{2}},
					\end{aligned}
				\end{equation}
				where $C_{4},C_{5}>0$ are constants which depend on $d,L,M_{0},M_{1},T$.
				
				Similarly, we also have the estimation for the integral on the set $B_{R,t}^{-} = \{x\in B_{R}:u(t,x)-u_{k}(t,x)\leq 0\}$:
				\begin{equation}
					\begin{aligned}
						\frac{d}{dt}\tilde{E}\biggl[1_{\Omega_{M_{1}}}&\int_{B_{R,t}^{-}}(u(t,x)-u_{k}(t,x))dx\biggr] \\&\leq C_{4}\tilde{E}\biggl[1_{\Omega_{M_{1}}}\int_{B_{R,t}^{-}}(u(t,x)-u_{k}(t,x))dx\biggr]
						+ C_{5}(t-\tau_{k-1})^{\frac{1}{2}},
					\end{aligned}
				\end{equation}
				and thus
				\begin{equation}
					\begin{aligned}
						\frac{d}{dt}\tilde{E}\biggl[1_{\Omega_{M_{1}}}&\int_{B_{R}}|u(t,x)-u_{k}(t,x)|dx\biggr] \\&\leq C_{4}\tilde{E}\biggl[1_{\Omega_{M_{1}}}\int_{B_{R}}|u(t,x)-u_{k}(t,x)|dx\biggr]
						+ 2C_{5}(t-\tau_{k-1})^{\frac{1}{2}}.
					\end{aligned}
				\end{equation}
				
				Therefore,
				\begin{equation}
					\begin{aligned} 
						\frac{d}{dt}\biggl(e^{-C_{4}(t-\tau_{k-1})}&\tilde{E}\biggl[1_{\Omega_{M_{1}}}\int_{B_{R}}|u(t,x)-u_{k}(t,x)|dx\biggr]\biggr)\\
						&\leq 2C_{5}e^{-C_{4}(t-\tau_{k-1})}(t-\tau_{k-1})^{\frac{1}{2}},
					\end{aligned}
				\end{equation}
				and
				\begin{equation}
					\begin{aligned}
						\tilde{E}\biggl[&1_{\Omega_{M_{1}}}\int_{B_{R}}|u(t,x)-u_{k}(t,x)|dx\biggr]\\
						&\leq e^{C_{4}(t-\tau_{k-1})}\left(\tilde{E}\biggl[1_{\Omega_{M_{1}}}\int_{B_{R}}|u(\tau_{k-1},x)-u_{k}(\tau_{k-1},x)|dx\biggr] + \frac{4}{3}C_{5}(t-\tau_{k-1})^{\frac{3}{2}}\right).
					\end{aligned}
				\end{equation}
				
				Notice that $u_{k}(\tau_{k-1},x)\equiv u_{k-1}(\tau_{k-1},x)$ by definition. Inductively, we have
				\begin{equation}
					\begin{aligned}
						\tilde{E}\biggl[&1_{\Omega_{M_{1}}}\int_{B_{R}}|u(\tau_{k},x)-u_{k}(\tau_{k},x)|dx\biggr]\\
						&\leq e^{C_{4}\delta}\left(\tilde{E}\biggl[1_{\Omega_{M_{1}}}\int_{B_{R}}|u(\tau_{k-1},x)-u_{k-1}(\tau_{k-1},x)|dx\biggr] + \frac{4}{3}C_{5}\delta^{\frac{3}{2}}\right)\\
						&\leq e^{C_{4}k\delta}\tilde{E}\biggl[1_{\Omega_{M_{1}}}\int_{B_{R}}|\sigma_{0}(x)-\sigma_{0}(x)|dx\biggr] + \frac{4}{3}C_{5}\delta^{\frac{3}{2}}\sum_{i=1}^{k}e^{C_{4}(i-1)\delta}\\
						&\leq \frac{4}{3}C_{5}\delta^{\frac{3}{2}}ke^{C_{4}k\delta}\leq C_{6}\delta^{\frac{1}{2}}.
					\end{aligned}
					\label{eq:15a}
				\end{equation}
				where $C_{6}$ is a constant which depends on $d,L,M_{0},M_{1},T$. 
				
				Also, for the value we are concerned with in \eqref{eq:64},
				\begin{equation}
					\begin{aligned}
						\tilde{E}\biggl[&1_{\Omega_{M_{1}}}\int_{B_{R}}e^{h^{\top}(x)Y_{\tau_{k}}}|u(\tau_{k},x)-u_{k}(\tau_{k},x)|dx\biggr]\\
						&\leq e^{M_{0}M_{1}}\tilde{E}\biggl[1_{\Omega_{M_{1}}}\int_{B_{R}}|u(\tau_{k},x)-u_{k}(\tau_{k},x)|dx\biggr]\leq C_{6}e^{M_{0}M_{1}}\delta^{\frac{1}{2}}
					\end{aligned}
				\end{equation}
				
				On the event $\Omega_{M_{1}}^{c} = \{\omega:\sup_{0\leq t\leq T}\sum_{j=1}^{d}|Y_{t,j}(\omega)|> M_{1}\}$, let
				$$\overline{Y}_{T}\triangleq\sup_{0\leq t\leq T}\sum_{j=1}^{d}|Y_{t,j}|,$$
				then, 
				\begin{equation}
					\begin{aligned} 
						\tilde{E}&\biggl[1_{\Omega_{M_{1}}^{c}}\int_{B_{R}}e^{h^{\top}(x)Y_{\tau_{k}}}|u(\tau_{k},x)-u_{k}(\tau_{k},x)|dx\biggr]\\
						&\leq \tilde{E}\biggl[1_{\Omega_{1,M_{1}}^{c}}\frac{\overline{Y}_{T}}{M_{1}}\exp\left(M_{0}\sum_{j=1}^{d}|Y_{\tau_{k},j}|\right)\int_{B_{R}}|u(\tau_{k},x) - u_{k}(\tau_{k},x)|dx\biggr]\\
						&\leq \frac{1}{M_{1}}\left(\tilde{E}\biggl[\overline{Y}_{T}^{2}\exp\left(2M_{0}\sum_{j=1}^{d}|Y_{\tau_{k},j}|\right)\biggr]\right)^{\frac{1}{2}}\left(\tilde{E}\left(\int_{B_{R}}|u(\tau_{k},x)-u_{k}(\tau_{k},x)|dx\right)^{2}\right)^{\frac{1}{2}}\\
						&\leq \frac{C_{7}}{M_{1}}\left(\tilde{E}\xi^{2}\right)^{\frac{1}{2}}\left(\tilde{E}\int_{B_{R}}|u(\tau_{k},x)|^{2}dx+\tilde{E}\int_{B_{R}}|u_{k}(\tau_{k},x)|^{2}dx\right)^{\frac{1}{2}}.
					\end{aligned}
					\label{eq:75}
				\end{equation}
				where $C_{7} > 0$ is a constant which is related to the volume of the $d$-dimensional ball $B_{R}$, and $\xi$ is the random variable given by
				\begin{equation}
					\xi = \overline{Y}_{T}\exp\left(M_{0}\sum_{j=1}^{d}|Y_{\tau_{k},j}|\right),
				\end{equation}
				and
				\begin{equation}
					\begin{aligned}
						\tilde{E}\xi^{2} = \tilde{E}\biggl[\overline{Y}_{T}^{2}\exp\left(2M_{0}\sum_{j=1}^{d}|Y_{\tau_{k},j}|\right)\biggr]\leq \left(\tilde{E}\overline{Y}_{T}^{4}\right)^{\frac{1}{2}}\left(\tilde{E}\exp\left(4M_{0}\sum_{j=1}^{d}|Y_{\tau_{k},j}|\right)\right)^{\frac{1}{2}}
					\end{aligned}
				\end{equation}
				
				According to the Burkholder-Davis-Gundy inequality (cf. \cite{karatzas1998brownian}, Chapter 3, Theorem 3.28, for example), there exists $C_{8} > 0$, such that
				\begin{equation}
					\tilde{E}\overline{Y}_{T}^{4}\leq C_{8}\tilde{E}\sum_{j=1}^{d}|Y_{T,j}|^{4}\leq 3C_{8}dT^{2}.
				\end{equation} 
				and also, because $Y_{\tau_{k},j}$ are normal random variables, the expectation of $$\exp\left(4M_{0}\sum_{j=1}^{d}Y_{\tau_{k},j}\right)$$ is bounded.
				
				For the value $\tilde{E}\int_{B_{R}}|u(\tau_{k},x)|^{2}dx$, because
				\begin{equation}
					u(t,x) = \exp\left(-\sum_{j=1}^{d}h_{j}(x)Y_{t,j}\right)\sigma(t,x),
				\end{equation}
				then
				\begin{equation}
					\begin{aligned} 
						\tilde{E}\int_{B_{R}}|u(\tau_{k},x)|^{2}dx&=\tilde{E}\int_{B_{R}}\exp\biggl(-2\sum_{j=1}^{d}h_{j}(x)Y_{\tau_{k},j}\biggr)\sigma^{2}(\tau_{k},x)dx\\
						&\leq \tilde{E}\biggl[\exp\left(2M_{0}\sum_{j=1}^{d}|Y_{\tau_{k},j}|\right)\int_{B_{R}}\sigma^{2}(\tau_{k},x)dx\biggr]\\
						&\leq\left(\tilde{E}\exp\left(4M_{0}\sum_{j=1}^{d}|Y_{\tau_{k},j}|\right)\right)^{\frac{1}{2}}\left(\tilde{E}\left(\int_{B_{R}}|\sigma(\tau_{k},x)|^{2}dx\right)^{2}\right)^{\frac{1}{2}}.
					\end{aligned}
				\end{equation}

				Notice that $\sigma(t,x)$ is the solution to the stochastic partial differential equation
				\begin{equation}
					d\sigma(t,x) = \mathcal{L}^{*}\sigma(t,x)dt + \sum_{j=1}^{d}h_{j}\sigma(t,x)dY_{t,j}.
				\end{equation}
				and the boundedness of
				\begin{equation}
					\tilde{E}\left(\int_{B_{R}}|\sigma(\tau_{k},x)|^{2}dx\right)^{2}
				\end{equation}
				follows from the regularity theory of stochastic partial differential equation. 
				
				In the monograph \cite{Stochastic}, the authors provided a similar regularity result, and proved that $\tilde{E}\int_{B_{R}}|\sigma(\tau_{k},x)|^{2}dx$ is bounded by the initial values. Here in our case, we will prove that there exists $C_{9}>0$, such that
				\begin{equation}
					\tilde{E}\left(\int_{B_{R}}|\sigma(\tau_{k},x)|^{2}dx\right)^{2}\leq C_{9}\left(\int_{B_{R}}|\sigma_{0}(x)|^{2}dx\right)^{2}.
					\label{eq:132}
				\end{equation}
				The detailed proof of \eqref{eq:132} can be found in the Appendix. 
				
				Therefore, we have
				\begin{equation}
					\tilde{E}\int_{B_{R}}|u(\tau_{k},x)|^{2}dx\leq C_{10},
				\end{equation}
				where $C_{10}> 0$ is a constant that does not depend on $\delta$ or $M_{1}$.
				
				Furthermore, as we have discussed in the previous section, $\tilde{E}\int_{B_{R}}|u_{k}(\tau_{k},x)|^{2}dx$ is also bounded above, and thus, we have 
				\begin{equation}
					\begin{aligned} 
						\tilde{E}\biggl[&1_{\Omega_{M_{1}}^{c}}\int_{B_{R}}e^{h^{\top}(x)Y_{\tau_{k}}}|u(\tau_{k},x)-u_{k}(\tau_{k},x)|dx\biggr]\leq \frac{C_{11}}{M_{1}},
					\end{aligned}
					\label{eq:82}
				\end{equation}
				where $C_{11}$ is a constant which does not depend on $M_{1}$ or $\delta$. 
				
				In summary, for each $\epsilon > 0$, there exists $M_{1}> 0$, such that
				\begin{equation}
					\frac{C_{11}}{M_{1}}<\frac{\epsilon}{2},
				\end{equation} 
				and for this particular $M_{1}$, there exists $\delta>0$, such that
				\begin{equation}
					C_{6}e^{M_{0}M_{1}}\delta^{\frac{1}{2}}<\frac{\epsilon}{2},
				\end{equation}
				Therefore, for every $k=1,\cdots, K$,
				\begin{equation}
					\begin{aligned} 
						\tilde{E}&\int_{B_{R}}e^{h^{\top}(x)Y_{\tau_{k}}}|u(\tau_{k},x) - u_{k}(\tau_{k},x)|dx \\
						&= \tilde{E}\biggl[1_{\Omega_{1,M_{1}}}\int_{B_{R}}e^{h^{\top}(x)Y_{\tau_{k}}}|u(\tau_{k},x)-u_{k}(\tau_{k},x)|dx\biggr] \\
						&\quad+ \tilde{E}\biggl[1_{\Omega_{1,M_{1}}^{c}}\int_{B_{R}}e^{h^{\top}(x)Y_{\tau_{k}}}|u(\tau_{k},x)-u_{k}(\tau_{k},x)|dx\biggr]\\
						&\leq \ C_{6}e^{M_{0}M_{1}}\delta^{\frac{1}{2}} + \frac{C_{10}}{M_{1}}<\epsilon.
					\end{aligned}
					\label{eq:84}
				\end{equation}
			\end{proof}
			\section{Conclusion}
			\label{sec:8}
			In this paper, we provide a novel convergence analysis of Yau-Yau algorithm from a probabilistic perspective. With very liberal assumptions only on the coefficients of the filtering systems and the initial distributions (without assumptions on particular paths of observations), we can prove that Yau-Yau algorithm can provide accurate approximations with arbitrary precision to a quite broad class of statistics for the conditional distribution of state process given the observations, which includes the most commonly used conditional mean and covariance matrix. Therefore, the capability of Yau-Yau algorithm to solve very general nonlinear filtering problems is theoretically verified in this paper.
			
			In the process of deriving this probabilistic version of the convergence results, we study the properties of the exact solution, $\{\sigma(t,x):0\leq t\leq T\}$, to the DMZ equation and the approximated solution $\{\tilde{u}_{k+1}(\tau_{k},x):1\leq k\leq K\}$, given by Yau-Yau algorithm, respectively. 
			
			For the exact solution $\sigma(t,x)$ of the DMZ equation, we have shown in Section \ref{sec:4} and Section \ref{sec:5} that most of the density of $\sigma(t,x)$ will remain in the closed ball $B_{R}$, and $\sigma(t,x)$ can be approximated well by the corresponding initial-boundary value problem of DMZ equation in $B_{R}$. This result also implies that it is very unlikely for the state process to reach infinity within finite terminal time. 
			
			For the approximated solution $\tilde{u}_{k+1}(\tau_{k},x)$ given by Yau-Yau algorithm, we have first proved in Section \ref{sec:6} that $\tilde{u}_{k+1}(\tau_{k},x)$, which evolves in a recursive manner, will not explode in finite time interval, even if the time-discretization step $\delta\rightarrow 0$. And then, in Section \ref{sec:7}, the convergence of $\tilde{u}_{k+1}(\tau_{k},x)$ is proved and the convergence rate is also estimated to be $\sqrt{\delta}$.
			
			It is clear that the properties of exact solutions and approximated solutions, which we have proved in this paper, highly rely on the nice properties of Brownian motion and Gaussian distributions, especially the Markov and light-tail properties. On the one hand, Brownian motion and Gaussian distribution are up to now, among the most commonly used objects in the mathematical modeling of many areas of applications, and can describe most scenarios in practice. On the other hand, for those systems driven by non-Markov or heavy-tailed processes, minimum mean square criteria, together with the conditional expectations (if exist), may not result in a satisfactory estimation of the state process. In this case, the studies of estimations based on other criteria, such as maximum a posteriori (MAP) \cite{godsill2001maximum}\cite{saha2012particle}\cite{kang2023log}, will be a promising direction.
			
			Finally, in this paper, we only consider filtering systems and conduct convergence analysis in time interval $[0,T]$ with a fixed finite terminal time $T$. It is also interesting to study the behavior of the DMZ equation and the approximation capability of Yau-Yau algorithm in the case where the terminal time $T\rightarrow\infty$, especially for filtering systems with further stable assumptions. We will continue working on how to combine the existing studies on filter stability, such as \cite{Atar1998}\cite{Ocone1996}, with our techniques developed in this paper, and hopefully, obtain some convergence results of Yau-Yau algorithm for the whole time line $(0,\infty)$. 

\section*{Declarations}
\subsection*{Funding}

This work is supported by National Natural Science
Foundation of China (NSFC) grant (12201631) and Tsinghua University Education Foundation fund (042202008).
\subsection*{Conflict of interest/Competing interests}
The authors have no competing interests to declare that are relevant to the content of this article.

\subsection*{Ethics approval and consent to participate}
Not applicable.
\subsection*{Consent for publication}
Not applicable.
\subsection*{Data, Materials and/or Code availability}
Not applicable.

\subsection*{Author contribution}
All authors contributed to the study conception and design. The first draft of the manuscript was written by Zeju Sun and all authors commented on previous versions of the manuscript. All authors read and approved the final manuscript.

%
%

%
%
%
%

\begin{appendices}

\section{Regularity Results of Parabolic Partial Differential Equation and Stochastic Evolution Equation} 
In this appendix, we will provide a detailed proof of the regularity results of the parabolic partial differential equation and the stochastic evolution equation. 

For the purpose of deriving \eqref{eq:110} and \eqref{eq:132}, the regularity results is slightly different from standard ones considered in square-integrable functional spaces.
\begin{theorem}
	Let $\sigma(t,x)$ be the solution of the following IBV problem: 
	\begin{equation}
		\left\{\begin{aligned}
			&\frac{\partial\sigma(t,x)}{\partial t} = \frac{1}{2}\sum_{i,j=1}^{d}\frac{\partial^{2}}{\partial x_{i}\partial x_{j}}(a^{ij}(x)\sigma(t,x)) - \sum_{i=1}^{d}\frac{\partial}{\partial x_{i}}(f_{i}(x)\sigma(t,x)) \\
			&\qquad- \frac{1}{2}|h(x)|^{2}\sigma(t,x),\ (t,x)\in[0,T]\times B_{R},\\
			&\sigma(0,x) =\sigma_{0}(x),\ x\in B_{R}\\
			&\sigma(t,x) = 0,\ (t,x)\in[0,T]\times\partial B_{R},
		\end{aligned}\right.
		\label{eq:98a}
	\end{equation}
	where $B_{R} = \{x\in\mathbb{R}^{d}:|x|\leq R\}$ is the ball in $\mathbb{R}^{d}$ with radius $R$; $a:\mathbb{R}^{d}\rightarrow\mathbb{R}^{d\times d}$, $f:\mathbb{R}^{d}\rightarrow\mathbb{R}^{d}$, $h:\mathbb{R}^{d}\rightarrow\mathbb{R}^{d}$ are smooth enough functions. Assume that the matrix-valued function $a(x)$ is uniformly positive definite, i.e., there exists $\lambda > 0$, such that
	\begin{equation}
		\sum_{i,j=1}^{d}a^{ij}(x)\xi_{i}\xi_{j}\geq \lambda|\xi|^{2},\ \forall\ x\in B_{R},\ \xi\in\mathbb{R}^{d}.
		\label{eq:137}
	\end{equation}
	If the initial value $\sigma_{0}(x)$ is quartic-integrable in $B_{R}$, then there exists a constant $C > 0$, which depends on the coefficients of the system, such that
	\begin{equation}
		\int_{B_{R}}\sigma^{4}(T,x)dx\leq e^{CT}\int_{B_{R}}\sigma_{0}^{4}(x)dx.
	\end{equation} 
\end{theorem}
\begin{remark}
	In fact, Assumption \textbf{(A2)} in the main text will imply the coercivity condition \eqref{eq:137}. This is because the closed ball $B_{R}$ is a compact set of $\mathbb{R}^{d}$, and the continuous function $\lambda(x)$ in Assumption \textbf{(A2)} will map $B_{R}$ to a compact set. Therefore, there exists $\lambda > 0$, such that $\lambda(x)\geq \lambda > 0$, for all $x\in B_{R}$. 
\end{remark}
\begin{proof}
	Let us define
	\begin{equation}
		\tilde{f}_{i}(x) = f_{i}(x) - \sum_{j=1}^{d}\frac{\partial a^{ij}(x)}{\partial x},\ i=1,\cdots,d.
	\end{equation}
	Then the parabolic equation \eqref{eq:98a} can be written in a divergence form
	\begin{equation}
		\frac{\partial\sigma(t,x)}{\partial t} = \frac{1}{2}\sum_{i,j=1}^{d}\frac{\partial}{\partial x_{i}}\left(a^{ij}(x)\frac{\partial}{\partial x_{j}}\sigma(t,x)\right) - \sum_{i=1}^{d}\frac{\partial}{\partial x_{i}}(\tilde{f}_{i}(x)\sigma(t,x))-\frac{1}{2}|h(x)|^{2}\sigma(t,x).
	\end{equation}
	Hence,
	\begin{equation}
		\begin{aligned}
			\frac{d}{dt}&\int_{B_{R}}\sigma^{4}(t,x)dx = \int_{B_{R}}4\sigma^{3}(t,x)\frac{\partial \sigma}{\partial t}dx\\
			&=\int_{B_{R}}2\sigma^{3}\sum_{i,j=1}^{d}\frac{\partial}{\partial x_{i}}\left(a^{ij}\frac{\partial}{\partial x_{j}}\sigma\right)dx -\int_{B_{R}}4\sigma^{3}\sum_{i=1}^{d}\frac{\partial}{\partial x_{i}}(\tilde{f}_{i}\sigma)dx - \int_{B_{R}}2\sigma^{4}|h|^{2}dx\\
			&= -6\int_{B_{R}}\sigma^{2}\sum_{i,j=1}^{d}a^{ij}\frac{\partial \sigma}{\partial x_{i}}\frac{\partial \sigma}{\partial x_{j}}dx + 12\int_{B_{R}}\sum_{i=1}^{d}\tilde{f}_{i}\sigma^{3}\frac{\partial\sigma}{\partial x_{i}}dx -2\int_{B_{R}}\sigma^{4}|h|^{2}dx\\
			&\leq -6\lambda\int_{B_{R}}\sigma^{2}|\nabla\sigma|^{2}dx + 12\int_{B_{R}}\sum_{i=1}^{d}\frac{\tilde{f}_{i}\sigma^{2}}{\sqrt{\lambda}}\cdot\left(\sqrt{\lambda}\sigma\frac{\partial\sigma}{\partial x_{i}}\right)dx -2 \int_{B_{R}}\sigma^{4}|h|^{2}dx\\
			&\leq -6\lambda\int_{B_{R}}\sigma^{2}|\nabla\sigma|^{2}dx + 12\int_{B_{R}}\sum_{i=1}^{d}\biggl(\frac{\tilde{f}_{i}^{2}\sigma^{4}}{2\lambda} + \frac{\lambda}{2}\sigma^{2}\biggl|\frac{\partial\sigma}{\partial x_{i}}\biggr|^{2}\biggr)dx - 2 \int_{B_{R}}\sigma^{4}|h|^{2}dx\\
			&\leq \int_{B_{R}}\biggl(\frac{6}{\lambda}\sum_{i=1}^{d}\tilde{f}_{i}^{2} - 2|h|^{2}\biggr)\sigma^{4}(t,x)dx.
		\end{aligned}
	\end{equation}
	In the bounded domain $B_{R}$, there exists a constant $C> 0$, such that
	\begin{equation}
		\biggl|\frac{6}{\lambda}\sum_{i=1}^{d}\tilde{f}_{i}^{2} - 2|h|^{2}\biggr|\leq C.
	\end{equation}
	Thus, 
	\begin{equation}
		\frac{d}{dt}\int_{B_{R}}\sigma^{4}(t,x)dx\leq C\int_{B_{R}}\sigma^{4}(t,x)dx,\ t\in [0,T],
	\end{equation}
	and by Gronwall's inequality, we have
	\begin{equation}
		\int_{B_{R}}\sigma^{4}(T,x)dx\leq e^{CT}\int_{B_{R}}\sigma_{0}^{4}(x)dx.
	\end{equation}
\end{proof}
\begin{theorem}
	Consider the IBV problem of stochastic partial differential equation given by
	\begin{equation}
		\left\{\begin{aligned} 
			&d\sigma(t,x) = \mathcal{L}^{*}\sigma(t,x)dt + \sum_{j=1}^{d}h_{j}(x)\sigma(t,x)dY_{t,j},\ t\in[0,T]\\
			&\sigma(t,x) = 0,\ (t,x)\in[0,T]\times\partial B_{R},\\
			&\sigma(0,x) = \sigma_{0}(x),\ x\in B_{R}.
		\end{aligned} \right.
		\label{eq:147}
	\end{equation}
	where $Y = \{Y_{t}:0\leq t\leq T\}$ is a standard $d$-dimensional Brownian motion in the filtered probability space $\left(\Omega,\mathcal{F},\{\mathcal{F}_{t}\}_{0\leq t\leq T}, P\right)$; $B_{R} = \{x\in\mathbb{R}^{d}:|x|\leq R\}$ is the ball in $\mathbb{R}^{d}$ with radius $R$, and
	\begin{equation}
		\mathcal{L}^{*}(\star) = \frac{1}{2}\sum_{i,j=1}^{d}\frac{\partial^{2}}{\partial x_{i}\partial x_{j}}(a^{ij}(x)\star) - \sum_{i=1}^{d}\frac{\partial}{\partial x_{i}}(f_{i}(x)\star). 
	\end{equation}
	Assume that the coefficients $a$, $f$, $h$ are smooth enough and the Assumption \textbf{(A2)} holds for the matrix-valued function $a(x)$, which implies that $a(x)$ is uniformly positive definite in $B_{R}$, i.e., there exists $\lambda > 0$, such that
	\begin{equation}
		\sum_{i,j=1}^{d}a^{ij}(x)\xi_{i}\xi_{j}\geq \lambda|\xi|^{2},\ \forall\ x\in B_{R},\ \xi\in\mathbb{R}^{d}.
	\end{equation}
	If the initial value $\sigma_{0}(x)$ is square-integrable in $B_{R}$, then there exists a constant $C > 0$, which depends on $T$, $R$ and the coefficients of the system, such that 
	\begin{equation}
		E\left(\int_{B_{R}}|\sigma(T,x)|^{2}dx\right)^{2}\leq C\left(\int_{B_{R}}|\sigma_{0}(x)|^{2}dx\right)^{2}.
		\label{eq:132a}
	\end{equation}
\end{theorem}
\begin{proof}
	Let us define
	\begin{equation}
		\tilde{f}_{i}(x) = f_{i}(x) - \sum_{j=1}^{d}\frac{\partial a^{ij}(x)}{\partial x},\ i=1,\cdots,d.
	\end{equation}
	Then the stochastic partial differential equation in \eqref{eq:147} can be rewritten in divergence form:
	\begin{equation}
		d\sigma(t,x) = \frac{1}{2}\sum_{i,j=1}^{d}\frac{\partial}{\partial x_{i}}\biggl(a^{ij}\frac{\partial\sigma}{\partial x_{j}}\biggr)-\sum_{i=1}^{d}\frac{\partial}{\partial x_{i}}(\tilde{f}_{i}\sigma) + \sum_{j=1}^{d}h_{j}\sigma dY_{t,j}.
	\end{equation} 
	Let
	\begin{equation}
		\Phi(t) = \int_{B_{R}}\sigma^{2}(t,x)dx,\ t\in[0,T],
	\end{equation}
	then according to It\^o's formula,
	\begin{equation}
		\begin{aligned}
			d\Phi(t) = \left(\int_{B_{R}}(2\sigma \mathcal{L}^{*}\sigma + \sigma^{2}|h|^{2})dx\right)dt + \sum_{j=1}^{d}\left(\int_{B_{R}}2h_{j}\sigma^{2}dx\right)dY_{t,j}
		\end{aligned}
	\end{equation}
	and
	\begin{equation}
		\begin{aligned} 
			d\Phi^{2}(t) &= 2\left(\int_{B_{R}}\sigma^{2}dx\right)\left(\int_{B_{R}}(2\sigma \mathcal{L}^{*}\sigma + \sigma^{2}|h|^{2})dx\right)dt \\
			&\quad+ 2\Phi(t)\sum_{j=1}^{d}\left(\int_{B_{R}}2h_{j}\sigma^{2}dx\right)dY_{t,j} + \sum_{j=1}^{d}\left(\int_{B_{R}}2h_{j}\sigma^{2}dx\right)^{2}dt
		\end{aligned} 
	\end{equation}
	After taking expectations, we have
	\begin{equation}
		\begin{aligned} 
			\frac{d}{dt}E\Phi^{2}(t) &= \frac{d}{dt}E\left(\int_{B_{R}}\sigma^{2}(t,x)dx\right)^{2}  \\
			&= E\biggl[2\left(\int_{B_{R}}\sigma^{2}dx\right)\left(\int_{B_{R}}(2\sigma \mathcal{L}^{*}\sigma + \sigma^{2}|h|^{2})dx\right)\\
			&\qquad+\sum_{j=1}^{d}\left(\int_{B_{R}}2h_{j}\sigma^{2}dx\right)^{2}\biggr]
		\end{aligned}
	\end{equation}
	Notice that
	\begin{equation}
		\begin{aligned} 
			\int_{B_{R}}2\sigma\mathcal{L}^{*}\sigma dx &= \int_{B_{R}}\sigma\sum_{i,j=1}^{d}\frac{\partial}{\partial x_{i}}\biggl(a^{ij}\frac{\partial\sigma}{\partial x_{j}}\biggr)dx - \int_{B_{R}}2\sigma\sum_{i=1}^{d}\frac{\partial}{\partial x_{i}}(\tilde{f}_{i}\sigma)dx\\
			&= - \int_{B_{R}}\sum_{i,j=1}^{d}a^{ij}\frac{\partial\sigma}{\partial x_{i}}\frac{\partial \sigma}{\partial x_{j}}dx + 2\int_{B_{R}}\sum_{i=1}^{d}\tilde{f}_{i}\sigma\frac{\partial \sigma}{\partial x_{i}}dx\\
			&\leq -\lambda\int_{B_{R}}|\nabla\sigma|^{2}dx + 2\int_{B_{R}}\sum_{i=1}^{d}\frac{\tilde{f}_{i}\sigma}{\sqrt{\lambda}}\cdot\biggl(\sqrt{\lambda}\frac{\partial\sigma}{\partial x_{i}}\biggr)dx\\
			&\leq -\lambda\int_{B_{R}}|\nabla\sigma|^{2}dx + \int_{B_{R}}\frac{1}{\lambda}\sum_{i=1}^{d}\tilde{f}_{i}^{2}\sigma^{2}dx + \lambda\int_{B_{R}}|\nabla\sigma|^{2}dx.
		\end{aligned}
	\end{equation}
	Hence,
	\begin{equation}
		\begin{aligned} 
			\frac{d}{dt}E\left(\int_{B_{R}}\sigma^{2}(t,x)dx\right)^{2}&\leq 2E\biggl[\biggl(\int_{B_{R}}\sigma^{2}dx\biggr)\biggl(\int_{B_{R}}\biggl(\frac{1}{\lambda}|\tilde{f}|^{2}+|h|^{2}\biggr)\sigma^{2}dx\biggr)\biggr]\\
			&\quad +E\biggl[\sum_{j=1}^{d}\left(\int_{B_{R}}2h_{j}\sigma^{2}dx\right)^{2}\biggr]
		\end{aligned} 
	\end{equation}
	In the bounded domain $B_{R}$, there exists $M > 0$, such that
	\begin{equation}
		\frac{1}{\lambda}|\tilde{f}(x)|^{2}+|h(x)|^{2}\leq M,\ |h_{j}(x)|\leq M,\ \forall \ x\in B_{R}.
	\end{equation}
	Thus,
	\begin{equation}
		\frac{d}{dt}E\left(\int_{B_{R}}\sigma^{2}(t,x)dx\right)^{2}\leq (2M+4dM^{2})E\left(\int_{B_{R}}\sigma^{2}(t,x)dx\right)^{2}
	\end{equation}
	According to Gronwall's inequality, 
	\begin{equation}
		E\left(\int_{B_{R}}\sigma^{2}(T,x)dx\right)^{2}\leq e^{(2M+4dM^{2})T}\left(\int_{B_{R}}\sigma_{0}^{2}(x)dx\right)^{2},
	\end{equation}
	which is the desired result.
\end{proof}

\end{appendices}



\end{document}